\documentclass[12pt]{amsart}

\usepackage{amssymb,amsfonts,amsthm,amsmath}
\usepackage{graphicx}
\usepackage{enumerate}
\usepackage[a4paper,ignoreall,margin=1in]{geometry}
\usepackage{blkarray}
\usepackage[colorlinks]{hyperref}
\usepackage[normalem]{ulem}

\newtheorem*{thm*}{Theorem}
\newtheorem{thm}{Theorem}[section]
\newtheorem{cor}[thm]{Corollary}
\newtheorem{prop}[thm]{Proposition}

\newtheorem{lemma}[thm]{Lemma}

\theoremstyle{remark}
\newtheorem{remark}[thm]{Remark}

\newtheorem{problem}{Problem}

\theoremstyle{definition}
\newtheorem{example}[thm]{Example}
\newtheorem{defn}[thm]{Definition}
\newtheorem{notation}[thm]{Notation}

\newcommand{\C}{\mathbb{C}}
\newcommand{\kk}{\Bbbk}
\newcommand{\A}{\mathbb{A}}
\newcommand{\PP}{\mathbb{P}}

\DeclareMathOperator{\rank}{rank}
\DeclareMathOperator{\codim}{codim}
\DeclareMathOperator{\image}{image}
\DeclareMathOperator{\Sing}{Sing}
\DeclareMathOperator{\mult}{mult}
\DeclareMathOperator{\Spec}{Spec}
\DeclareMathOperator{\length}{length}
\DeclareMathOperator{\DirSum}{DirSum}
\DeclareMathOperator{\Con}{Con}
\DeclareMathOperator{\Sub}{Sub}
\DeclareMathOperator{\ApoEqu}{ApoEqu}
\DeclareMathOperator{\ApoLim}{ApoLim}
\DeclareMathOperator{\Tor}{Tor}
\DeclareMathOperator{\Soc}{Soc}

\newcommand{\defining}[1]{\emph{#1}}

\newcommand{\smarthook}{\ensuremath{\mathbin{\text{\raisebox{.4ex}{%
  \vrule height .5pt width 1ex depth 0pt%
  \vrule height 0.8ex width .5pt depth 0pt%
}}\mathchoice{}{}{\mkern3mu}{\mkern3mu}}}}

\newcommand{\apolarityaction}{\smarthook}
\renewcommand{\aa}{\apolarityaction}

\newcommand{\veronese}{v}

\newcommand{\set}[1]{\left\{#1\right\}}
\newcommand{\fromto}[2]{#1, \dotsc, #2}
\newcommand{\setfromto}[2]{\set{\fromto{#1}{#2}}}

\newcommand{\Wedge}[1]{{\textstyle{\bigwedge\nolimits}^{\! #1}}}

\newcommand{\gotm}{\mathfrak{m}}
\newcommand{\gotp}{\mathfrak{p}}

\title[Apolarity and direct sum decomposability]{Apolarity and direct sum decomposability of polynomials}
 \author[W.~Buczy\'nska]{Weronika Buczy\'nska}\thanks{W.~Buczy\'nska is supported 
   by the research project
  ``Rangi i rangi brzegowe wielomian\'ow oraz r\'ownania rozmaito\'sci siecznych''
   funded by Polish Financial Means for Science in 2012-2014. }
 \address{Weronika Buczy\'nska\\
 Institute of Mathematics of the
 Polish Academy of Sciences\\
 ul. \'Sniadeckich 8\\
 P.O. Box 21\\
 00-956 Warszawa, Poland}
 \email{wkrych@mimuw.edu.pl}
 \author[J.~Buczy\'nski]{Jaros\l{}aw Buczy\'nski}
 \thanks{J.~Buczy\'nski is supported by the project
        ``Secant varieties, computational complexity, and toric degenerations''
        realised within the Homing Plus programme of Foundation for Polish Science,
        cofinanced from European Union, Regional Development Fund.
        J.~Buczy\'nski is also supported by the scholarship ``START'' of the Foundation for Polish Science.}
 \address{Jaros\l{}aw Buczy\'nski\\
 Institute of Mathematics of the
 Polish Academy of Sciences\\
 ul. \'Sniadeckich 8\\
 P.O. Box 21\\
 00-956 Warszawa, Poland}
 \email{jabu@mimuw.edu.pl}
\author[J.~Kleppe]{Johannes Kleppe}
\address{Johannes Kleppe\\
Department of Technology\\
Buskerud University College\\
Norway}
\email{johannes.kleppe@gmail.com}
\author[Z.~Teitler]{Zach Teitler}
\address{Zach Teitler\\
Department of Mathematics \\
1910 University Drive \\
Boise State University \\
Boise, ID 83725-1555 \\
USA}
\email{zteitler@boisestate.edu}
\date{\today}

\begin{document}

\bibliographystyle{amsplain}       

\begin{abstract}
A polynomial is a direct sum
  if it can be written as a sum of two non-zero polynomials in some distinct sets of variables, 
  up to a linear change of variables.
We analyse criteria for a homogeneous polynomial to be decomposable as a direct sum,
  in terms of the apolar ideal of the polynomial.
We prove that the apolar ideal of a polynomial of degree $d$ strictly depending on all variables 
  has a minimal generator of degree $d$  if and only if it is a limit of direct sums.
\end{abstract}
\maketitle

\section{Introduction}
A homogeneous polynomial $F$ is a \defining{direct sum}
if there exist non-zero polynomials $F_1$, $F_2$ such that
$F = F_1 + F_2$ and $F_1 = F_1(t_1,\dotsc,t_s)$, $F_2 = F_2(t_{s+1},\dotsc,t_n)$
for some linearly independent linear forms $t_1,\dotsc,t_n$.
For example, $F=xy$ is a direct sum, as $F = \frac{1}{4}(x+y)^2 - \frac{1}{4}(x-y)^2$.
In coordinate-free terms, $F \in S^d V$ is a direct sum if $F = F_1 + F_2$
for nonzero $F_i \in S^d V_i$, $i = 1,2$, such that $V_1 \oplus V_2 = V$.

Most polynomials are not direct sums,
see Lemma~\ref{lemma: general indecomposable}.
Nevertheless it can be difficult to show that a particular polynomial is not a direct sum.
For instance, Sepideh Shafiei shared with us the following question:
is the generic determinant $\det_n = \det((x_{i,j})_{i,j=1}^n)$,
a homogeneous form of degree $n$ in $n^2$ variables, a direct sum?
For $n=2$, $\det_2 = x_{1,1}x_{2,2} - x_{1,2}x_{2,1}$ is visibly a direct sum.
On the other hand, for $n>2$ it is easy to see the determinant is not decomposable as a direct sum in the original variables,
but it is not immediately clear whether it is decomposable after a linear change of coordinates.
We answer this question in the negative, see Corollary~\ref{cor_determinant_is_not_dir_sum}.
\begin{problem}\label{problem: dir sum}
Give necessary or sufficient conditions for a polynomial to be a direct sum.
\end{problem}

We approach this problem through \textit{apolarity}.
Suppose $S = \C[x_1,\dotsc,x_n]$ and $T = \C[\alpha_1,\dotsc,\alpha_n]$.
When the number of variables is small we may write $S = \C[x,y]$ and $T = \C[\alpha, \beta]$,
or $S = \C[x,y,z]$ and $T = \C[\alpha, \beta, \gamma]$.
(For simplicity we assume throughout that our base field is
  the field of complex numbers $\C$.
However, our results also hold for other algebraically closed base fields of any characteristic.
We comment on the applicable modifications in
Section~\ref{sect: other fields}.)
We let $T$ act on $S$ by letting $\alpha_i$ act as the partial differentiation operator $\partial/\partial x_i$.
This action is denoted by the symbol $\aa$,
as in $\alpha \beta^2 \aa x^2 y^3 z^4 = \partial^3 x^2 y^3 z^4 / \partial x \partial y^2 = 12 x y z^4$.
This is the \defining{apolarity action}; $T$ is called the \defining{dual ring} of $S$.
Let $F \in S$ be a homogeneous polynomial of degree $d$.
The \defining{apolar} or \defining{annihilating ideal} $F^\perp \subset T$
is the set of polynomials $\Theta \in T$ such that $\Theta \aa F = 0$.
The quotient $A_F = T/F^\perp$ is called the \defining{apolar algebra} of $F$.

The \defining{Waring rank} $r(F)$ of $F$ is the least $r$ such that
$F = \ell_1^d + \dotsb + \ell_r^d$ for some linear forms $\ell_i$.
A lower bound for Waring rank, following from ideas of Sylvester in 1851 \cite{Sylvester:1851kx},
is that $r(F)$ is bounded below by the maximum value of the Hilbert function of $A_F$.
Ranestad and Schreyer \cite{MR2842085}
have recently shown that the Waring rank of $F$ is bounded below by $\frac{1}{\delta} \length(A_F)$,
where $\delta$ is the greatest degree of a minimal generator of $F^\perp$
and $\length(A_F)$ is the length of the apolar algebra,
that is, the sum of all the values of the Hilbert function of $A_F$.
The bound of Ranestad--Schreyer is best when $\delta$ is small,
that is when $F^\perp$ is generated in small degrees.
So it is natural to ask when this occurs,
or conversely when $F^\perp$ has high-degree generators.

\begin{problem}\label{problem: apolar gen degrees}
Give necessary or sufficient conditions for $F^\perp$ to be generated in low degrees or in high degrees;
that is, for the greatest degree $\delta$ of a minimal generator of $F^\perp$ to be small or large.
\end{problem}
As we shall see, ``small'' and ``large'' should be considered relative to the degree $d$ of $F$.

It is through serendipity that while simultaneously studying
Problems \ref{problem: dir sum} and \ref{problem: apolar gen degrees}, 
as separate problems,
the authors noticed that they were actually not separate.
These two problems are linked by the following result (see also \cite[Lem.~2.9, Lem.~3.27]{KleppePhD2005}).

\begin{thm}\label{thm: direct sum generator}
If $F$ is a direct sum then $F^\perp$ has a minimal generator of degree $\deg(F)$.
\end{thm}
Sepideh Shafiei has shown that the apolar ideal of the generic determinant $\det_n$
is generated in degree $2$ \cite{Shafiei:ud}.
Thus
\begin{cor}\label{cor_determinant_is_not_dir_sum}
For $n > 2$ the generic determinant is not a direct sum.
\end{cor}
Other results of Shafiei concerning apolar ideals of permanents, Pfaffians, etc., have
similar consequences for direct sum indecomposability of these forms.

Despite its centrality in linking
Problems~\ref{problem: dir sum} and~\ref{problem: apolar gen degrees},
Theorem~\ref{thm: direct sum generator} is surprisingly easy to prove,
see Section~\ref{sect: max deg apolar gen of direct sum}.

The converse to Theorem~\ref{thm: direct sum generator} does not hold.
\begin{example}\label{example: border rank 2 apoequ not dirsum}
$F = xy^2 \in S = \C[x,y]$ has
   $F^\perp = \langle \alpha^2, \beta^3 \rangle \subset T = \C[\alpha, \beta]$
with the minimal generator $\beta^3$ of degree $3$, but $F$ is not a direct sum.
Indeed, in two variables a direct sum $x^d - y^d$ factors as $x^d - y^d = \prod_{k=1}^d (x - \zeta^k y)$
($\zeta$ a primitive $d$-th root of unity), with distinct linear factors,
while $xy^2$ does not have distinct factors; or use Proposition~\ref{prop: smith stong}.
\end{example}

\begin{example}\label{example: plane cubic apoequ not dirsum}
The cubic $F = x^2 y - y^2 z = y(x^2 - yz) \in \C[x,y,z]$
has $F^\perp = \langle \gamma^2, \alpha\gamma, \alpha^2 + \beta\gamma, \beta^3, \alpha\beta^2 \rangle$,
so $F^\perp$ has two minimal generators of degree $3$.
Thus $F$ satisfies the necessary condition of Theorem~\ref{thm: direct sum generator}.
However $F$ is not a direct sum by Proposition~\ref{prop: smith stong}.
\end{example}

Note however that $xy^2$ is a limit of direct sums:
\[
  xy^2 = \lim_{t \to 0} \frac{1}{3t} \Big( (y+tx)^3 - y^3 \Big)
\]
as is $y(x^2+yz)$:
\[
  y(x^2+yz) = \lim_{t \to 0} \frac{1}{6t^2} \Big( (y + tx + 2t^2 z)^3 + (y - tx)^3 - 2y^3 \Big) .
\]
We will show that if $F^\perp$ has a minimal generator of degree $\deg(F)$,
then $F$ is a limit of direct sums.
But the converse does not hold: not every limit $F$ of direct sums has the property that
$F^\perp$ has a minimal generator of degree $\deg(F)$.

\begin{example}\label{example: nonconcise limit of dirsum}
For $t \neq 0$, $x^d - ty^d$ is a direct sum
and $\lim_{t \to 0} x^d - ty^d = x^d$.
However $(x^d)^\perp = \langle \alpha^{d+1},\beta \rangle$
has no minimal generator of degree $d$.

A perhaps more satisfying example is $\lim_{t \to 0} xyz - tw^3 = xyz$,
again a limit of direct sums,
with $(xyz)^\perp = \langle \alpha^2, \beta^2, \gamma^2, \delta \rangle$,
having no minimal generator of degree $3$.
\end{example}
It is no coincidence that in both of these examples the limit polynomial
uses fewer variables than the direct sums at $t \neq 0$.
We will show that in general, if $F$ is a limit of direct sums which cannot be written
using fewer variables, then $F^\perp$ has a minimal generator of degree $\deg(F)$.

We now introduce terminology to give a precise statement of these results.

First note that
$F^\perp$ is a homogeneous ideal containing $\langle \alpha_1,\dotsc,\alpha_n \rangle^{d+1}$,
all forms of degree at least $d+1$.
Thus $F^\perp$ is generated in degree at most $d+1$:
the $\delta$ in the Ranestad--Schreyer theorem satisfies $1 \leq \delta \leq d+1$.
We mention the following observation,
previously noted by Casnati and Notari \cite[Rem.~4.3]{MR2769229}.
\begin{prop}
\label{prop: d+1 generator rank 1}
$F^\perp$ has a minimal generator of degree $d+1$ if and only if $F = \ell^d$ is a power of a linear form.
\end{prop}
A proof is given in Section~\ref{sect: max deg apolar gen of direct sum}.

For brevity we refer to a minimal generator of $F^\perp$ as an \defining{apolar generator} of $F$.
Any apolar generator of degree equal to $\deg(F)$ is called an \defining{equipotent apolar generator}.

We introduce the notation $\DirSum = \DirSum_{n;d}$
for the set of direct sums (of degree $d$ in $n$ variables),
$\ApoEqu$ for the set of forms with an equipotent apolar generator,
and $\Con$ for the set of forms that cannot be written using fewer variables.
(Such forms are called \defining{concise}, see Section~\ref{sect: conciseness}.)
We will show that every form with an equipotent apolar generator
is a limit of direct sums, so that
we have the following inclusions:
\[
\begin{array}{ccccc}
  \DirSum &\subset& \ApoEqu &\subset& \overline{\DirSum} \\
  
  \cup & & \cup & & \cup \\

  \DirSum \cap \Con &\subset& \ApoEqu \cap \Con &\subset& \overline{\DirSum} \cap \Con
\end{array}
\]
In fact most of these inclusions are strict in general.
The vertical inclusions clearly are strict as soon as $n \geq 2$.
We have $\DirSum \cap \Con \subsetneqq \ApoEqu \cap \Con$
(and of course $\DirSum \subsetneqq \ApoEqu$)
by Examples \ref{example: border rank 2 apoequ not dirsum} and \ref{example: plane cubic apoequ not dirsum}.
And we have $\ApoEqu \subsetneqq \overline{\DirSum}$
by Example \ref{example: nonconcise limit of dirsum}.

Surprisingly, the last remaining inclusion is in fact an equality (compare with \cite[Cor.~4.7]{KleppePhD2005}).
\begin{thm}\label{thm: apoequ = closure of dirsum}
For $n \geq 2$ and $d \geq 3$,
every form with an equipotent apolar generator
is a limit of direct sums and conversely,
every concise limit of direct sums has an equipotent apolar generator.
In particular $\ApoEqu \cap \Con = \overline{\DirSum} \cap \Con$.
\end{thm}
The theorem is proved in Section~\ref{sect: maximal degree apolar gens and limits}.
  One direction is proved in Theorem~\ref{thm: apoequ => limit of dirsum} 
  and the other direction is proved in Theorem~\ref{thm: limit of s fold direct sum then s-1 max deg apo gens}.
Moreover, Theorem~\ref{thm: apoequ => limit of dirsum} 
   provides a normal form for the limits of direct sums which are not direct sums.
   In such cases, for some choice of basis $x_1,\dotsc,x_k$, $y_1,\dotsc,y_k$, $z_1,\dotsc,z_{n-2k}$ of $V$:
   \begin{multline}
     F(\fromto{x_1}{x_k},\fromto{y_1}{y_k},\fromto{z_1}{z_{n-2k}}) \\
        = \sum_{i=1}^k x_i \frac{\partial H(\fromto{y_1}{y_k})}{\partial y_i} + G(\fromto{y_1}{y_k},\fromto{z_1}{z_{n-2k}})
        \label{equ_normal_form_for_limit_of_dir_sums}
   \end{multline}
   for homogeneous polynomials $H(y)$ in $k$ variables and $G(y,z)$ in $n-k$ variables, both of degree $d$.
  
One might hope naively to prove at least one direction of Theorem~\ref{thm: apoequ = closure of dirsum}
by arguing that if $F_t \to F$, then presumably $F_t^\perp \to F^\perp$.
If for each $t \neq 0$, $F_t$ is a direct sum, then $F_t^\perp$ has a minimal generator of degree $d = \deg F$
by Theorem~\ref{thm: direct sum generator};
and then one might hope to finish by appealing to the semicontinuity of graded Betti numbers,
to show that $F^\perp$ also has at least one minimal generator of degree $d$.
However this argument cannot succeed,
as $F_t \to F$ does not imply $F^\perp_t \to F^\perp$ as a flat limit.
For instance, consider the family of polynomials $F_t = tx^d + xy^{d-1}$
in $x$ and $y$ parametrized by~$t$, with $d \ge 4$.
We have $F_t \to F_0 = xy^{d-1}$, and
\[
 F_t^{\perp} =
 \begin{cases}
    \langle    \alpha^2\beta, \alpha^{d-1} + dt\beta^{d-1} \rangle & \text{for $t\ne 0$, or} \\
    \langle    \alpha^2, \beta ^d \rangle                          & \text{for $t = 0$.}
 \end{cases}
\]
Thus the flat limit $\lim_{t\to 0} (F_t^{\perp}) = \langle \alpha^2\beta, \alpha^{d-1}, \beta^d \rangle \subsetneqq F_0^{\perp}$.

Nevertheless, for those cases in which $F_t \to F$, the $F_t$ are direct sums, and $F^\perp_t \to F^\perp$ is a flat family,
it follows that $F^\perp$ has a degree $d$ generator by semicontinuity.
When such a family $\{F_t\}$ exists, we say $F$ is an \defining{apolar limit} of direct sums.
The locus of apolar limits of direct sums is denoted $\ApoLim$.
We have
\begin{thm}\label{thm: apolim and apoequ}
$\ApoLim \subset \ApoEqu$, and:
\begin{enumerate}
 \item There exists $n$ such that for $d \ge 2n$,
       the inclusion is strict: $\ApoLim \subsetneqq \ApoEqu$.
 \item If $d=3$ or if $n=3$, then $\ApoLim = \ApoEqu$.
\end{enumerate}
\end{thm}
In other words, the inclusion is strict for some $n$ and sufficiently large $d$,
but there is equality in some cases, including $d=3$ or $n=3$.
The inclusion $\ApoLim \subset \ApoEqu$
follows by the semicontinuity of graded Betti numbers, as described above.
The existence of $n$ for which the inclusion is strict is explained in Section~\ref{sect: non-apolar limit}, 
   particularly, in Proposition~\ref{prop: limit not apolar limit}.
The proof of the equality for $d=3$ is straightforward, and it is explained in Proposition~\ref{prop: cubic apolar limit}.
The proof of the equality for $n=3$ is obtained by longer,
   but elementary methods in Theorem~\ref{thm: apolim=apoequ in the plane}.
Certainly, using more refined techniques one may be able 
   to determine whether $\ApoLim = \ApoEqu$ also in other cases.
In any case, we emphasize that because of Theorem~\ref{thm: apolim and apoequ},
the naive hope described above cannot suffice to prove Theorem~\ref{thm: apoequ = closure of dirsum}.
That is, for some $n$ and $d$, there are forms in $\ApoEqu$ whose apolar generators of degree $d$
do not arise via semicontinuity of graded Betti numbers for any family of direct sums.
The strictness of $\ApoLim \subsetneqq \ApoEqu$
forces a more delicate argument for Theorem~\ref{thm: apoequ = closure of dirsum}.

The strictness of the inclusion is a consequence of the existence of certain zero-dimen\-sio\-nal Gorenstein local schemes
with restricted deformations, which we call \emph{uncleavable schemes}. 
A scheme supported at a single point is uncleavable if all its deformations are supported at a single point,
see Section~\ref{sect: non-apolar limit} for references and more details.
We show that for at least $n=14$ we have $\ApoLim\subsetneqq \ApoEqu$.
This is because the shortest non-smoothable zero-dimensional Gorenstein scheme of length $14$ is uncleavable.
However, we expect that $\ApoLim \subsetneqq \ApoEqu$ should hold for all sufficiently large $n$.
  We explain in detail in Section~\ref{sect: non-apolar limit} 
  which deformation theoretic properties of schemes of length $n$ we need in order to obtain $\ApoLim \neq \ApoEqu$. 

See also \cite[Sect.~4.2, Cor.~4.24]{KleppePhD2005} for related examples.

\medskip

It is also interesting to study the case in which $F^\perp$ is generated in low degrees.
For example, $(\det_n)^\perp$ is generated in degree $2$,
as is $(x_1\dotsm x_n)^\perp = \langle \alpha_1^2,\dotsc,\alpha_n^2 \rangle$.
See Table~\ref{table:plane cubics} for examples of plane cubics.
We show that an upper bound for the degrees of minimal generators of $F^\perp$
forces an upper bound on the degree of $F$;
equivalently, if $F$ has a high degree relative to the number of variables
then $F^\perp$ must have high degree minimal generators.
\begin{thm}\label{thm: apolar generator degree lower bound}
If $F$ is a homogeneous form of degree $d$ in $n$ variables
and $\delta$ is the highest among the degrees of minimal generators of $F^\perp$
then $d \leq (\delta-1)n$.
\end{thm}
In particular if $F^\perp$ is generated by quadrics then $d \leq n$.
It would be interesting to classify polynomials $F$
of degree $d=n$ such that $F^\perp$ is generated by quadrics.
See Section~\ref{sect: lower bound on degree} for a brief discussion and the proof of the theorem.

\begin{notation}
Throughout the paper, $F \in S^d V$ is a homogeneous form of degree $d$
in $n = \dim V$ variables.
More generally $F$ may be a divided-powers form of degree $d$, see \cite[App.~A]{MR1735271}.
\end{notation}

\begin{remark}
Direct sums and their limits have also appeared in other articles.
In \cite{MR1067383}, functions (not necessarily polynomials)
are called \emph{decomposable} when they are sums of functions in independent variables.
In \cite{MR1096431} they are called \emph{sum-maps}, 
   while in \cite{MR1127057} and \cite{MR1969308} they are called \emph{direct sums}.
In \cite{MR2548229}, polynomials with a direct sum decomposition are called \emph{polynomials of Sebastiani--Thom type}.
They are called \emph{connected sums} in \cite{Shafiei:ud},
following \cite{MR2177162}, \cite{MR2738376}, where the term \emph{connected sum} is used
to refer to a closely related concept, see Section \ref{sect: connected sums}.
In \cite{MR1119265}, forms (homogeneous polynomials) $p$ and $q$ over $\C$
are called \emph{unitarily disjoint} if they depend on disjoint sets of variables,
after a unitary linear change of variables 
   with respect to a fixed Hermitian product on the space of linear forms.
   (see \cite{MR1119265} for details).
In \cite{KleppePhD2005},
direct sum decompositions are called \emph{regular splittings}
and limits of direct sum decompositions are called \emph{degenerate splittings}.

In \cite{khoury_jayanthan_srinivasan} and references therein the authors study apolar algebras of homogeneous forms $F$,
  which are either direct sums $F=x^d + G(y_1,\dotsc, y_{n-1})$ of ``type'' $(1,n-1)$ 
  or their limits $x y^{d-1} + G(y, z_1,\dotsc, z_{n-2})$ 
    --- compare with the normal form \eqref{equ_normal_form_for_limit_of_dir_sums} for $k=1$.
Their work is motivated by earlier articles \cite{MR2158748},
  \cite{elkhoury_srinivasan_a_class_of_Gorenstein_Artin_algebra_of_codim_4},
  where the special case of $n=4$ has been studied.
In this series of articles the direct sums and their limits serve the purpose 
  of a classification of Gorenstein Artin algebras with prescribed invariants.
Our results in this article may have similar applications, which need to be further studied.
\end{remark}

\subsection{Outline of paper}
In the remainder of this Introduction we give proofs of some elementary statements including
Theorem~\ref{thm: direct sum generator} and Proposition~\ref{prop: d+1 generator rank 1}.

In Section~\ref{sect: background} we review background, including:
apolarity; conciseness; secant varieties and border rank;
the easy cases of binary forms and plane cubics;
semicontinuity of graded Betti numbers;
Gorenstein Artin algebras;
and connected sums.

In Section~\ref{S: dimension of dirsum}
we discuss the dimension of the direct sum locus and uniqueness of direct sum decompositions.

In Section~\ref{sect: apolar generators and limits of direct sums}
we collect results that relate quadratic apolar generators
to direct sums and to maximal degree apolar generators.
We prove Theorem~\ref{thm: apoequ = closure of dirsum}.
Then we prove Theorem~\ref{thm: apolar generator degree lower bound}.

In Section~\ref{sect: variation in families}
we prove Theorem~\ref{thm: apolim and apoequ}.

In Section~\ref{sect: generalizations}
we generalize some of our results to linear series of forms (instead of a single form).
We consider ``almost direct sums.''
Finally we discuss the generalization of our results to algebraically closed fields in any characteristic.

\subsection{Equipotent apolar generator of a direct sum}\label{sect: max deg apolar gen of direct sum}

We begin with a few elementary statements.

\begin{proof}[Proof of Theorem~\ref{thm: direct sum generator}]
Say $F = G-H$ where $G \in S^x = \C[x_1,\dotsc,x_i]$, $H \in S^y = \C[y_1,\dotsc,y_j]$,
and $G, H \neq 0$.
Let us denote the dual rings $T^\alpha = \C[\alpha_1,\dotsc,\alpha_i]$,
$T^\beta = \C[\beta_1,\dotsc,\beta_j]$.
We work in $S = S^x \otimes S^y = \C[x_1,\dotsc,x_i,y_1,\dotsc,y_j]$
with dual ring
$T = T^\alpha \otimes T^\beta = \C[\alpha_1,\dotsc,\alpha_i,\beta_1,\dotsc,\beta_j]$.

We have $G^\perp \cap H^\perp \subset F^\perp$,
where $G^\perp$ and $H^\perp$ are computed in $T$ rather than $T^\alpha$, $T^\beta$.
On the other hand if $\Theta \in (F^\perp)_k$,
then $\Theta \aa G = \Theta \aa H \in S^x_{d-k} \cap S^y_{d-k}$.
But this intersection is zero if $k \neq d$, so we must have $\Theta \aa G = \Theta \aa H = 0$.
Thus $(G^\perp \cap H^\perp)_k = (F^\perp)_k$ for all $k \neq d$.

Now let $\delta_1 \in T^\alpha_d$ such that $\delta_1 \aa G = 1$
and let $\delta_2 \in T^\beta_d$ such that $\delta_2 \aa H = 1$.
Such elements exist in abundance: there is an affine hyperplane of them in
$T^\alpha_d$ and in $T^\beta_d$, by the hypothesis that $G$ and $H$ are nonzero.
Let $\Delta = \delta_1 + \delta_2$.
Then $\Delta \aa G = \Delta \aa H = 1$, so $\Delta \notin G^\perp \cap H^\perp$,
but $\Delta \aa F = 0$.

This element $\Delta$ is a minimal generator of $F^\perp$:
it cannot be generated in lower degrees, since all elements in lower degrees
lie in $G^\perp \cap H^\perp$.
\end{proof}

For future reference we record the additional details
given in the above proof (see also \cite[Lem.~3.27]{KleppePhD2005}).

\begin{lemma}\label{lem: apolar of direct sum}
   Let $F = G-H$ be a direct sum decomposition of degree $d$, with $G$, $H$ nonzero.
   Then 
   \[
      F^{\perp} =  G^\perp \cap H^\perp + \langle \Delta \rangle   
   \]
   where $\Delta = \delta_1 + \delta_2 \in T_d$ is homogeneous of degree $d$,
   $\delta_1 \aa G = \delta_2 \aa H = 1$,
   $\delta_1$ can be written only using variables dual to variables of $G$,
   and $\delta_2$ can be written only using variables dual to variables of $H$.
\end{lemma}
\begin{proof}
The only statement left to prove is that any degree $d$ element of $F^{\perp}$
 is in the ideal $G^\perp \cap H^\perp + \langle \Delta \rangle$.
Let $\Theta \in (F^\perp)_d$.
Then $\Theta \aa G = \Theta \aa H \in \C$; call this value $c$.
We have $\Theta - c\Delta \in G^\perp \cap H^\perp$.
\end{proof}
   
See also \cite[Lemma~3.1]{Casnati:2013ad} for a description of $F^\perp$
in terms of the extensions of the ideals $G^\perp \cap T^\alpha$, $H^\perp \cap T^\beta$.

\begin{cor}
If $F = F_1 + \dotsb + F_s$ is a direct sum of $s$ terms
then $F^\perp$ has at least $s-1$ equipotent apolar generators.
\end{cor}
\begin{proof}
It follows by induction on $s$ from Lemma~\ref{lem: apolar of direct sum}.
Explicitly, if $F = F_1 + \dotsb + F_s$ then
\begin{multline*}
  \dim F^\perp_d = 1 + \dim(F_1^\perp \cap (F_2 + \dotsb + F_s)^\perp)_d \\
    = \dotsb = s-1 + \dim (F_1^\perp \cap \dotsb \cap F_s^\perp)_d ,
\end{multline*}
so $F^\perp$ has at least $s-1$ minimal generators of degree $d$.
\end{proof}
We call $F$ an \defining{$s$-fold direct sum} when $F$ can be written
as a direct sum of $s$ terms, that is, $F = F_1 + \dotsb + F_s$ with each $F_i \in V_i$
and $V_1 \oplus \dotsb \oplus V_s = V$.

We will use frequently the following simple characterization of direct sums.
\begin{cor}\label{cor_quadratic_generators_of_a_direct_sum}
Let $V = V_1 \oplus V_2$ and $F \in S^d V$. 
Then the following are equivalent:
\begin{enumerate}
 \item \label{item_a_direct_sum}
       $F = F_1 + F_2$ where $F_1 \in S^d V_1$ and $F_2 \in S^d V_2$,
 \item \label{item_quadratic_generators}
        $V_1^* V_2^* \subset F^\perp$,
 \item \label{item_zero_locus}
        $V_1 \cup V_2$  contains the common affine scheme-theoretic zero locus $V((F^\perp)_2)$ of the quadrics in $F^{\perp}$.
\end{enumerate}
\end{cor}
Note that in \ref{item_a_direct_sum} $F$ is not necessarily a direct sum, as $F_1$ or $F_2$ could be zero.
It is easy to overcome this, by simply adding an assumption that $F^{\perp}$ has no linear generators and both vector spaces $V_1$ and $V_2$ are non-trivial.
\begin{proof}
   If $F= F_1 + F_2$, then clearly the reducible quadrics in  $V_1^* V_2^*$ annihilate $F$.
   In the other direction, if  $V_1^* V_2^* \subset F^\perp$, and we give $V_1$ a basis 
      $\fromto{x_1}{x_a}$ and $V_2$ a basis $\fromto{y_1}{y_{n-a}}$,
      then the condition \ref{item_quadratic_generators} implies that
      $F$ cannot have mixed terms divisible by $x_i y_j$. Thus $F$ is as in \ref{item_a_direct_sum}.
      
   Finally \ref{item_zero_locus} is simply a geometric rephrasing of \ref{item_quadratic_generators},
      using the correspondence between ideals and affine schemes.
\end{proof}

 We give an alternate proof of the statement
 observed by Casnati and Notari \cite[Rem.~4.3]{MR2769229},
 that a form $F$ of degree $d$ has an apolar generator of degree $d+1$
 if and only if $F = x^d$ has Waring rank $1$.
This proof illustrates in a simple case some of the techniques we will use later.
\begin{proof}[Proof of Proposition~\ref{prop: d+1 generator rank 1}]
If $F = x_1^d$ then $F^\perp = \langle \alpha_1^{d+1},\alpha_2,\dotsc,\alpha_n \rangle$.

Conversely suppose $F$ has degree $d$ and $F^\perp_{d+1}$ has a minimal generator.
Let $I = (F^\perp)_{\leq d}$, the ideal generated by forms in $F^\perp$ of degree at most $d$,
  and note that $I_{d+1}\subset T_{d+1} = F^\perp_{d+1}$
  has codimension at least $1$, because otherwise no generator would be needed. 
Then there is a nonzero polynomial $G$ of degree $d+1$ annihilated by $I$,
since $G$ is annihilated by $I$ if and only if it is annihilated by $I_{d+1}$ 
(see \cite[Prop.~3.4(iii)]{MR3121848}).
Moreover $I_{d} = (F^\perp)_{d}$ has codimension exactly $1$,
by the symmetry of the Hilbert function of $A_F$ (see \textsection\ref{section: gorenstein artin}):
$\dim (A_F)_d = \dim (A_F)_0 = 1$.
Now $I_d \subset (G^\perp)_d \subsetneqq T_d$,
so $I_d = (G^\perp)_d$.
Thus the Hilbert function of $A_G$ has $h_{A_G}(d) = 1$.
By symmetry $h_{A_G}(1) = 1$.
Then $G$ has only one essential variable (see \textsection\ref{sect: conciseness})
so $G$ can be written as a homogeneous form of a single variable;
necessarily $G$ has Waring rank $1$.
Say $G = x^{d+1}$.
We have $G^\perp\subset F^{\perp}$, since $(G^\perp)_{\le d} = I_{\le d} = (F^{\perp})_{\le d}$ 
  and $(F^{\perp})_{\ge d+1} = T_{\ge d +1}$.
So $F = \alpha \aa G = c x^d$ for some $\alpha$ by Lemma~\ref{lemma: apolar containment},
  i.e., by the inclusion-reversing part of the Macaulay inverse system Theorem.
\end{proof}

\subsection*{Acknowledgements}

As indicated by appropriate citations, 
  many of the statements in this article are either a part of the third author's PhD thesis \cite{KleppePhD2005}, 
  or they are similar to statements in that thesis.
We thank Sepideh Shafiei for suggesting to us the problem of direct sum decomposability
of determinants and for sharing with us her work in preparation.
We thank Joachim Jelisiejew, Grzegorz Kapustka, Jan O.~Kleppe, Robert Lazarsfeld, and Anna Otwinowska for helpful comments,
  and Kristian Ranestad for connecting the unpublished thesis of the third author with the work of the other authors.
We especially thank the anonymous referee for
unusually thorough reviews
  which included a number of helpful suggestions.
The computer algebra software packages Macaulay2 \cite{M2} and Magma \cite{magma}
were useful in compiling Table~\ref{table:plane cubics} and in calculations of examples.

\section{Background}\label{sect: background}

For a homogeneous ideal $I$ in the polynomial ring $S$,
a \defining{minimal generator} of $I$ is a non-zero homogeneous element of the graded module $I/\mathfrak{m}I$
where $\mathfrak{m} = \langle x_1,\dotsc,x_n \rangle$ is the irrelevant ideal.
By the ``\defining{number}'' of minimal generators of a given degree $k$
we mean the dimension of the $k$-th graded piece $(I/\mathfrak{m}I)_k$.

Following the convention of \cite{hartshorne}, by an \emph{(algebraic) variety}, we always mean an irreducible algebraic set.
By a \defining{general element} of an algebraic variety we always mean any element of some suitably chosen open dense subset.

When $V$ is a vector space, $\PP V$ denotes the projective space of lines through the origin of $V$.
When $v \in V$ is a nonzero vector,
$[v]$ denotes the point in $\PP V$ determined by $v$, that is, the line through the origin of $V$ spanned by $v$.

\subsection{Apolarity}

Let $S$ be a polynomial ring and $T$ its dual ring.
For a fixed homogeneous $F \in S_d$ of degree $d$,
the $i$-th \defining{catalecticant} $C^i_F$ is a linear map $T_{d-i} \to S_i$
defined by $C^i_F(\Theta) = \Theta \aa F$.
The term ``catalecticant'' was introduced by Sylvester in 1851~\cite{Sylvester:1851kx}.
The images of the catalecticants are the inverse systems studied by Macaulay~\cite{MR1281612}.

The catalecticant maps give an isomorphism between
$A_G = T/G^\perp$ and the principal $T$-submodule of $S$ generated by $G$,
consisting of elements $\Theta \aa G$ for $\Theta \in T$.
\begin{lemma}\label{lemma: apolar containment}
Suppose $F, G \in S$ are two homogeneous polynomials.
If $G^\perp \subset F^{\perp}$, then $F = \Theta \aa G$ for some $\Theta \in T$. 
\end{lemma}
Indeed, by the inclusion-reversing part of Theorem 21.6 of \cite{eisenbud:comm-alg},
the $T$-submodule of $S$ generated by $F$ is contained in the $T$-submodule generated by $G$.

One connection between apolarity and geometry is indicated by
Exercise 21.6 of \cite{eisenbud:comm-alg},
which relates the apolar ideals of plane conics to their ranks.
Another connection is given by the following well-known lemma (see for example \cite[Proposition 4.1]{MR1215329}).
\begin{lemma}\label{lem: apolarity singularity}
Let $\alpha \in T_1$ be a linear form.
Then $\alpha^k \in F^\perp$ if and only if $F$ vanishes to order at least $d-k+1$
at the corresponding point in the projective space $[\alpha] \in \PP T_1 \cong \PP S_1^*$.
\end{lemma}
In particular $\alpha^{d-1} \in F^\perp$ if and only if $V(F)$ is singular at $[\alpha]$.
\begin{proof}
$\alpha^k \in F^\perp$ is equivalent to $\Theta \alpha^k \aa F = 0$ for all $\Theta \in T_{d-k}$,
equivalently $\alpha^k \aa (\Theta \aa F) = 0$ for all $\Theta \in T_{d-k}$.
For such $\Theta$, $\Theta F$ is a form of degree $k$,
so $\alpha^k \aa (\Theta \aa F)$ is equal to the evaluation of $\Theta \aa F$ at the point $[\alpha]$ (up to scalar multiple).
This vanishes for all $\Theta \in T_{d-k}$ precisely when $F$ vanishes at $[\alpha]$
to order at least $d-k+1$.
\end{proof}

More detailed treatments of apolarity may be found in
\cite[Lect.~8]{Geramita},
\cite[Sect.~1.1]{MR1735271},
and
\cite{MR3121848}.

\subsection{Conciseness}\label{sect: conciseness}

A homogeneous form $F \in \C[x_1,\dotsc,x_n]$
is \defining{concise} (with respect to $x_1,\dotsc,x_n$)
if $F$ cannot be written as a polynomial in fewer variables.
That is, if there are linearly independent linear forms $t_1,\dotsc,t_k$ such that
$F \in \C[t_1,\dotsc,t_k] \subset \C[x_1,\dotsc,x_n]$,
then $k = n$.
In coordinate-free terms, $F \in S^d V$ is concise (with respect to $V$)
if $F \in S^d W$, $W \subset V$ implies $W = V$.

Concise polynomials are also called \defining{nondegenerate},
but we will follow the terminology of the tensor literature.

The following are equivalent:
\begin{enumerate}
\item $F \in S^d V$ is concise.
\item The hypersurface $V(F) \subset \PP V^*$ is not a cone.
\item There is no point in $\PP V^*$ at which $F$ vanishes to order $d$.
\item The catalecticant $C^1_F$ is onto.
\item The apolar ideal $F^\perp$ has no linear elements: $F^\perp_1 = 0$.
\end{enumerate}

We define the \defining{span} of $F$, denoted $\langle F \rangle$,
to be the image of the catalecticant $C^1_F$.
We have $F \in S^d \langle F \rangle$.
With this notation, $G + H$ is a direct sum decomposition if and only if
$\langle G \rangle \cap \langle H \rangle = \{0\}$ and $G,H \neq 0$.
The elements of $\langle F \rangle$ are called \defining{essential variables} of $F$;
by the number of essential variables of $F$ we mean the dimension
$\dim \langle F \rangle$ of $\langle F \rangle$ as a $\C$-vector space.
See \cite{MR2279854}.

The locus in $S^d V$ of non-concise polynomials
is a Zariski-closed subset called the \defining{subspace variety} and denoted $\Sub$.
Its complement is the open set $\Con$.

\subsection{Secant varieties and border rank}

Let $\veronese_d\colon \PP V \to \PP(S^d V )$ be the Veronese map, $\veronese_d([\ell]) = [\ell^d]$.

Recall $F \in S^d V $ has Waring rank $r$ if and only if $F$ is a sum of $r$ $d$-th powers
of linear forms in $V $, but not fewer.
Equivalently, $[F] \in \PP (S^d V )$ lies in the linear span of some $r$ points in the Veronese
variety $\veronese_d(\PP V )$, but does not lie in the span of any fewer points.
The Zariski closure of the set of projective points
corresponding to affine points of rank at most $r$ is the $r$-th secant variety
$\sigma_r(\veronese_d(\PP V ))$ of the Veronese variety.
The \defining{border rank} of $F$, denoted $br(F)$,
is the least $k$ such that $[F]$ lies in the $k$-th secant variety
of the Veronese variety.
Evidently $br(F) \leq r(F)$ and strict inequality may occur.

Note $\dim \langle F \rangle \leq br(F)$.
Indeed, $\dim \langle F \rangle \leq r(F)$ clearly,
so $\dim \langle F \rangle \leq r$ for all $F$ in a dense subset of $\sigma_r(\veronese_d(\PP V ))$.
Since $\dim \langle F \rangle = \rank C^1_F$ varies lower semicontinuously in $F$, 
we have $\dim \langle F \rangle \leq r$ for all $F$ in $\sigma_r(\veronese_d(\PP V ))$.

The second secant variety $\sigma_2(\veronese_d(\PP V ))$ is the disjoint union
of the set of points of rank $2$,
the set $\veronese_d(\PP V )$ itself,
and (for $d>2$) the set of points on tangent lines to $\veronese_d(\PP V )$.
Points of the third type have border rank $2$,
so only $2$ essential variables.
Such a point necessarily has the form $x y^{d-1}$
after a linear change of variables; we have $r(x y^{d-1}) = d$.
Thus $br(F)=2$ if and only if either $F=x^d+y^d$ and $r(F)=2$,
or $F=x y^{d-1}$ and $r(F)=d$.

We remark that the extreme case of direct sum, i.e., the $n$-fold direct sum in an $n$-dimensional vector space $V$ coincides with a sufficiently general element  of the $n$-th secant variety $\sigma_n(\veronese_d(\PP V))$.
In particular, the closure of the set of such extreme direct sums is equal to the secant variety.

\subsection{Binary forms}

The following lemma is standard; see, for example, Theorem 1.44 of \cite{MR1735271}.
\begin{lemma}
The apolar ideal of a homogeneous binary form $F$ of degree $d$
is a complete intersection ideal generated in degrees $r$ and $d+2-r$
for some integer $1 \leq r \leq (d+2)/2$.
The border rank of $F$ is $r$.
\end{lemma}

\begin{cor}
Let $F$ be a binary form of degree $d$.
The apolar ideal of $F$ has a generator of degree $d$
if and only if $F$ has border rank $2$.
\end{cor}
Note that the condition $br(F)=2$ excludes polynomials of rank $1$,
so $F$ must be concise.
Thus the locus of concise forms with an equipotent degree apolar generator
is exactly the locus of concise forms which are limits of direct sums,
that is, $\ApoEqu \cap \Con = \overline{\DirSum} \cap \Con$.
This is the case $n=2$ of Theorem~\ref{thm: apoequ = closure of dirsum}.

\subsection{Plane cubics}\label{sect: plane cubics}

If a plane cubic $F$ is a direct sum then in suitable coordinates
we may write $F = x^3 + G(y,z)$ where $G$ is a nonzero binary cubic form.
We may choose coordinates so that $G(y,z)$ is $y^3$, $y^3+z^3$, or $y^2 z$,
that is, $r(G) = 1$, $2$ or $3$.
Thus up to change of coordinates there are exactly three plane cubics which
are direct sums.

We summarize the types of plane cubics in Table~\ref{table:plane cubics},
adapted from \cite{Landsberg:2009yq}.
The columns mean the following:
$\beta_{1,i}$ is the number of apolar generators of degree $i$,
$r$ is Waring rank, and $br$ is border rank.
(We omit $\beta_{1,4}=1$ for $F = x^3$.)
The rows representing direct sums are in \textbf{bold face}
and the rows representing non-concise polynomials are in \textit{italic face}.

\begin{table}[hbt]
\begin{center}
\begin{tabular}{l l c l l l l}
Description & normal form & $\beta_{1,1}$ & $\beta_{1,2}$ & $\beta_{1,3}$ & $r$ & $br$ \\
\hline\hline
\textit{triple line} & $x^3$ & $2$ & $0$ & $0$ & $1$ & $1$ \\
\hline
\textit{\textbf{three concurrent lines}} & $x^3-y^3$ & $1$ & $1$ & $1$ & $2$ & $2$ \\
\hline
\textit{double line + line} & $x^2y$ & $1$ & $1$ & $1$ & $3$ & $2$ \\
\hline
\textbf{irreducible (Fermat)} & $x^3 + y^3 + z^3$ &  & $3$ & $2$ & $3$ & $3$ \\
\hline
irreducible & $y^2z - x^3 - xz^2$ & & $3$ & $0$ & $4$ & $4$ \\
\hline
\textbf{cusp} & $y^2z - x^3$ & & $3$ & $2$ & $4$ & $3$ \\
\hline
triangle & $xyz$ & & $3$ & $0$ & $4$ & $4$ \\
\hline
conic + transversal line & $x(x^2+yz)$ & & $3$ & $0$ & $4$ & $4$ \\
\hline
irreducible, smooth & $y^2z - x^3$ & & $3$ & $0$ & $4$ & $4$ \\
\qquad ($a^3 \neq 0, -27/4$) & \qquad $- axz^2 - z^3$ \\
\hline
irreducible, singular & $y^2z - x^3$ & & $3$ & $0$ & $4$ & $4$ \\
\qquad ($a^3 = -27/4$) & \qquad $- axz^2 - z^3$ \\
\hline
conic + tangent line & $y(x^2 + yz)$ & & $3$ & $2$ & $5$ & $3$ \\
\hline
\end{tabular}
\end{center}
\caption{Plane cubic curves.}\label{table:plane cubics}
\end{table}

This table shows the case $n=3$, $d=3$ of Theorem~\ref{thm: apoequ = closure of dirsum}.
\begin{cor}
Let $F$ be a concise plane cubic.
The apolar ideal of $F$ has a minimal generator of degree $3$
if and only if $F$ is a limit of direct sums.
\end{cor}
\begin{proof}
Table \ref{table:plane cubics} shows that a concise plane cubic has a minimal apolar generator
of degree $3$ if and only if the cubic has border rank $3$, which is equivalent to its
being a limit of Fermat cubics.
\end{proof}

\subsection{Semicontinuity of graded Betti numbers}\label{sect: semicontinuity of graded Betti numbers}

In this section we work over an arbitrary algebraically closed field $\kk$.
Let $I$ be a homogeneous ideal in a polynomial ring $T =\kk[\alpha_1,\dotsc,\alpha_n]$ with standard grading.
The \defining{graded Betti numbers} of $I$ are defined as follows.
Fix a minimal free resolution of $T/I$,
\[
  0 \leftarrow T \leftarrow \bigoplus T(-j)^{\beta_{1,j}}
    \leftarrow \bigoplus T(-j)^{\beta_{2,j}}
    \leftarrow \dotsb
\]
The $\beta_{i,j}$ are the graded Betti numbers of $I$ (more precisely, of $T/I$).
We have $\beta_{i,j} = \dim_{\kk} \Tor^i(I,\kk)_{j}$
\cite[Prop.~1.7]{eisenbud:syzygies}.

In the proof of Theorem~\ref{thm: limit of s fold direct sum then s-1 max deg apo gens}
we will use the fact that when the ideal $I$ varies in a flat family, the graded Betti numbers
vary upper-semicontinuously.
That is, if $I_t$ is a flat family of ideals, $\beta_{i,j}(I_0) \geq \lim_{t \to 0} \beta_{i,j}(I_t)$.

Boraty\'nski and Greco proved that when the ideal $I$ varies in a flat family, the Hilbert functions and Betti numbers
vary semicontinuously \cite{MR839041}.
Ragusa and Zappal\'a proved the semicontinuity of graded Betti numbers of
flat families of zero-dimensional ideals
\cite[Lem.~1.2]{MR2165191}.
Semicontinuity of graded Betti numbers more generally seems to be a well-known ``folk theorem'';
for example, different ideas for proofs are sketched in \cite[Remark following Theorem~1.1]{MR2084070},
and in \cite[Corollary~3.3]{MR3060753}.
We give a quick proof here for the sake of self-containedness.

\begin{prop}\label{prop: semicontinuity of graded Betti numbers}
Let $T = \kk[\alpha_1,\dotsc,\alpha_n]$ with standard grading,
and consider the power series ring $\kk[[U]]$, with $\deg U = 0$.
Suppose $I \subset T \otimes_{\kk} \kk[[U]]$ is a homogeneous ideal, flat over $\Spec(\kk[[U]])$.
For $\gotp \in \Spec(\kk[[U]])$ let $I_{\gotp} = I \otimes k(\gotp)$.
Fix any $i$ and $j$.
Then the function $\gotp \mapsto \beta_{i,j}(I_\gotp)$ is upper-semicontinuous.
\end{prop}

\begin{proof}
Start with the Koszul resolution of $\kk = T/(\alpha_1,\dotsc,\alpha_n)$,
regarded as a sheaf on $\Spec (\kk[[U]])$ (although independent of $U$).
Tensor the resolution with $I$,
take the degree $j$ part of the resulting complex, 
and denote by $I_{k, \gotp}$ the $k$-th graded piece of $I_{\gotp}$.
The $\Tor$ we are interested in is the homology of this complex of vector spaces:
\[
  \dotsb \leftarrow \Wedge{i-1} V^{*} \otimes I_{j-i+1,\gotp}
  \leftarrow \Wedge{i} V^{*} \otimes I_{j-i,\gotp}
  \leftarrow \Wedge{i+1} V^{*} \otimes I_{j-i-1,\gotp}
  \leftarrow \dotsb
\]
where $V^{*}$ is the vector space spanned by $\alpha_1,\dotsc,\alpha_n$.
By \cite[Exer.~20.14]{eisenbud:comm-alg},
the dimensions of the vector spaces $I_{q,\gotp}$ are locally (in $\gotp$) constant.
Locally in $\gotp$, then, this is a complex of fixed finite dimensional vector spaces
with differentials given by matrices whose entries are polynomial in $\gotp$.
The graded Betti number $\beta_{i,j}$ is the dimension of the $i$-th cohomology
of this complex; the dimensions of cohomology of such complexes are upper semicontinuous.
\end{proof}

\begin{remark}
Graded Betti numbers of flat families of ideal sheaves on projective space are \emph{not} semicontinuous.
For example, let three points in $\PP^2$ move from linearly independent position for $u \neq 0$
to collinear position when $u=0$.
For $u\neq0$, the ideal sheaf $\tilde I_u$ is generated by three quadrics having two linear syzygies.
At $u=0$ the ideal sheaf $\tilde I_0$ is a complete intersection of type $(1,3)$
(with one linear generator, one cubic generator, and just one syzygy).

The point is that the sheaf $\tilde I_0$ is the sheafification of the flat limit ideal $I_0$.
In the above example the flat limit ideal has an embedded point at the origin, which is lost in the sheafification.
\end{remark}

Our brief proof does not recover the ``consecutive cancellation''
as in \cite{MR2084070}, but we will not use consecutive cancellation.

\subsection{Gorenstein Artin algebras}\label{section: gorenstein artin}

Let $A$ be an algebra. 
Most of the time we will consider \defining{standard graded algebras}, 
  that is, $A$ is a graded algebra with  $A_0 = \C$ and $A$ is generated in degree $1$.
In this situation, 
  the \defining{embedding dimension} of $A$ is $\dim A_1$.
Let $\gotm = \bigoplus_{i>0} A_i$ be the graded maximal ideal.
The \defining{socle} of a graded algebra $A$ is the ideal $\Soc(A) = (0 : \gotm)$,
that is, the annihilator of the graded maximal ideal in $A$.
When $A$ is Artinian the socle includes $A_d$ where $d = \max\{i : A_i \neq 0\}$.
When $A$ is Artinian, $A$ is \defining{Gorenstein} if and only if $\Soc(A)$
is $1$-dimensional.
The \defining{socle degree} of $A$ is $\max\{i : A_i \neq 0 \}$.

We use \cite[Cor.~21.16]{eisenbud:comm-alg}.
Say $F$ is a concise homogeneous form of degree $d$ in $n$ variables
and $I = F^\perp$ is a zero-dimensional Gorenstein ideal,
so $A = T/I$ is a Gorenstein Artin algebra.
Then $A$ has socle degree $d = \deg F$ and embedding dimension $n$.
Let $A = T/I$ have the minimal free resolution $M_{\bullet}$:
\[
  0  \leftarrow T\! = \! M_0 \stackrel{d_1}{\leftarrow} M_1 \stackrel{d_2}{\leftarrow} \dotsb \stackrel{d_{n}}{\leftarrow} M_n  \leftarrow 0.
\]
The resolution $M_{\bullet}$ is self-dual, that is,
isomorphic to its dual, up to shifts in grading and homological degrees.
We call this isomorphism the \defining{Gorenstein symmetry}.
In particular,
writing each $M_i = \bigoplus_j T(-a^i_j)$, we have:
\[
  M_n = M_0^* = T(-d-n) \quad \text{and} \quad M_{n-i} = M_i^* = \bigoplus T(-d-n+a^i_j).
\]

The main focus of this paper is Gorenstein ideals having a minimal generator in degree $d$, that is $\beta_{1,d}(I) > 0$.
Throughout the article we will frequently use the following consequence of the Gorenstein symmetry:
\begin{equation}\label{equ: Gorenstein symmetry for maximal degree apo gens}
   \beta_{1,d}(F^{\perp}) = \beta_{n-1,n}(F^{\perp}), \text{ thus } \beta_{1,d}(F^{\perp}) > 0 \Longleftrightarrow \beta_{n-1,n}(F^{\perp}) >0.
\end{equation}
As we shall see, $\beta_{n-1,n}(F^{\perp})$ can be easier to control than $\beta_{1,d}(F^{\perp})$.

We will also use the more elementary symmetry of the Hilbert function of a graded Gorenstein Artin algebra,
$h_A(i) = h_A(d-i)$ for a Gorenstein Artin algebra $A$ of socle degree $d$.
See for example \cite[Theorem~4.1]{MR0485835}.

We will make use of the following two results.
The first is a special case of Thm.~8.18 of \cite{eisenbud:syzygies}.
\begin{lemma}[{\cite[Thm.~8.18]{eisenbud:syzygies}}]\label{lem: beta gt 0 then minors}
Suppose $I \subset T$ is a homogeneous ideal with $\beta_{n-1,n}(I) > 0$ and no linear generators.
Then there exists a choice of coordinates $\fromto{\alpha_1}{\alpha_n}$ of $T$
and linearly independent linear forms $\fromto{\ell_1}{\ell_k} \in T_1$ for some $0<k<n$
such that the $2\times 2$ minors of the following matrix are contained in $I$:
\begin{equation}\label{equ: n times 2 matrix with linear forms}
  \begin{pmatrix}
    \alpha_1 & \cdots & \alpha_k & \alpha_{k+1}& \cdots & \alpha_n \\
    \ell_1   & \cdots & \ell_k   &   0         & \cdots & 0 \\
  \end{pmatrix}.
\end{equation}
\end{lemma}
The second is a special case of Thm.~8.11 of \cite{eisenbud:syzygies}.
\begin{lemma}[{\cite[Thm.~8.11]{eisenbud:syzygies}}]\label{lem: beta inequality subideal}
Suppose $I \subset T$ is a homogeneous ideal containing no linear forms
and $J \subset I$ is a homogeneous subideal.
Then $\beta_{n-1,n}(J) \leq \beta_{n-1,n}(I)$.
\end{lemma}

In Section~\ref{sect: variation in families} we will also mention Gorenstein Artin algebras
  which are not necessarily graded.
  More precisely we will consider finite Gorenstein schemes that are spectra of those algebras.
  These schemes arise naturally when treating deformations of Gorenstein Artin schemes.

\subsection{Connected sum}\label{sect: connected sums}

When $A$, $A'$ are graded Gorenstein Artin algebras over a field $\Bbbk$,
both of socle degree $d$,
the (formal) \defining{connected sum} $A \# A'$ is defined as follows
\cite{MR2177162,MR2738376}.
$A \# A'$ is the graded algebra with graded pieces
\[
  (A \# A')_k =
    \begin{cases}
      \Bbbk, & k=0, \\
      A_k \oplus A'_k, & 0 < k < d, \\
      \Bbbk & k=d,
    \end{cases}
\]
in which the products of two elements in $A$ or in $A'$ are as before modulo
the identification of $A_d \cong A'_d \cong (A \# A')_d$,
and the product of a positive-degree element in $A$ with one in $A'$ is zero.
(See also \cite{MR2929675,Ananthnarayan:2014gf} for more general constructions.)

That this is named the ``connected sum'' of algebras is motivated by the following example.
If $X, Y$ are $d$-dimensional connected closed manifolds with cohomology rings $A_X$, $A_Y$,
then the cohomology ring of the connected sum $X \# Y$ is the connected sum of the cohomology rings:
$A_{X \# Y} = A_X \# A_Y$.

When a polynomial is a direct sum as we have defined it,
its apolar algebra is a connected sum in the above sense 
   (see also \cite[Lem.~3.27]{KleppePhD2005}).
\begin{prop}\label{prop: direct sum poly -> connected sum apolar algebra}
If $F = G - H$ is a direct sum decomposition then $A_F = A_{G} \# A_{H}$.
\end{prop}
\begin{proof}
Let $d = \deg F$, $T$, $T^\alpha$, and $T^{\beta}$ be as in the proof of Lemma~\ref{lem: apolar of direct sum}.

By Lemma~\ref{lem: apolar of direct sum} the annihilators satisfy 
   $(G^\perp)_k \cap (H^\perp)_k = (F^\perp)_k$ when $k < d$.
Note that $T^\alpha_1 \subset H^\perp$ and $T^\beta_1 \subset G^\perp$.
Thus for $p,q>0$, $T^\alpha_p \otimes T^\beta_q \subset G^\perp \cap H^\perp \subset F^\perp$.
Recall that $G^\perp$ is the apolar ideal of $G$ in $T$, i.e., of $G$ as an element of $S$;
the apolar ideal of $G$ in $T^\alpha$ (considering $G \in S^x$) is $G^\perp \cap T^\alpha$,
and similarly for $H$.
Hence for $0 < k < d$,
\begin{multline*}
  (A_F)_k
    = T_k / F^\perp_k
    = \Big( \bigoplus_{p+q=k} T^\alpha_p \otimes T^\beta_q \Big) / F^\perp \\
    = (T^\alpha)_k/(G^\perp \cap T^\alpha)_k \oplus (T^\beta)_k/(H^\perp \cap T^\beta)_k
    = (A_{G})_k \oplus (A_{H})_k
\end{multline*}
as claimed.
\end{proof}

We can use this to give a simple ``toy'' application of our results.
Suppose $X$ and $Y$ are $d$-dimensional connected closed complex manifolds
with cohomology rings $A_X$, $A_Y$, and suppose that these rings are standard graded
(which is by no means typical: cohomology rings of manifolds can contain generators of different degrees).
Write $A_X = S^x/G^\perp$ and $A_Y = S^y/H^\perp$.
Then the connected sum $X \# Y$ has cohomology ring $A_{X \# Y} \cong A_X \# A_Y \cong (S^x \otimes S^y)/(G+H)^\perp$.
Therefore if $M$ is a $d$-dimensional connected closed complex manifold
whose cohomology ring $A_M = S/F^\perp$ is standard graded and $F$ is not decomposable as a direct sum,
then $M$ is not decomposable as a connected sum, at least not into factors whose cohomology rings are standard graded.
In particular if $F^\perp$ has no minimal generator in degree $d$ then this holds by Theorem~\ref{thm: direct sum generator}.

There are well-known topological consequences of a direct sum decomposition,
for example involving monodromy \cite{MR0293122} and logarithmic vector fields \cite{MR2548229}.
It is not immediately obvious what \emph{geometric} consequences
may follow from a direct sum decomposition.
R.~Lazarsfeld shared with the fourth author the observation that
if $F = F_1 + F_2$ is a direct sum then
$\Sing(V(F)) = \Sing(V(F_1)) \cap \Sing(V(F_2))$,
that is, the singular locus of $F$ is an intersection of two cones with disjoint vertices.
Furthermore, defining $\Sigma_a(G) = \{ p \mid \mult_p(G) > a \}$,
the common zero locus of the $a$-th partial derivatives of $G$
(so that $\Sigma_0(G) = V(G)$, $\Sigma_1(G) = \Sing V(G)$),
we have
$\Sigma_a(F) = \Sigma_a(F_1) \cap \Sigma_a(F_2)$ for all $a>0$.

One necessary condition for $F$ to be a direct sum
can be deduced immediately from Proposition~4.2 of \cite{MR2738376}, which we state here for the reader's convenience.
We use the following terminology, taken from \cite{MR2738376}:
A \defining{standard graded Poincar\'e duality algebra} of \defining{formal dimension $d$} is precisely
a (standard graded) Gorenstein Artin algebra of socle degree $d$
(together with a choice of a nonzero socle element, which we ignore).
The \defining{rank} of such an algebra $H$ is the dimension of $H_1$.
The \defining{$\times$-length} of a subspace $V \subset H_1$
is the least integer $c$ such that any product of $c+1$ elements
of $V$ is zero in $H$ if such an integer exists, otherwise the $\times$-length of $V$ is infinite.
In particular, $V$ has $\times$-length strictly less than $d$ if and only if any product of $d$ elements of $V$ is zero in $H$.
\begin{prop}[{Proposition~4.2 of \cite{MR2738376}}]\label{prop: smith and stong original}
Let $H$ be a standard graded Poincar\'e duality algebra of formal dimension $d$.
Suppose there is a codimension one subspace $V \subset H_1$ of $\times$-length strictly less than $d$.
Then, either
\begin{enumerate}
\item $H$ is indecomposable with respect to the connected sum operation $\#$, or
\item $H$ has rank two and $H \cong \mathbb{F}[x,y]/(xy,x^d-y^d) \cong (\mathbb{F}[x]/(x^{d+1})) \# (\mathbb{F}[y]/(y^{d+1}))$.
\end{enumerate}
\end{prop}
Note that we work over the base field $\mathbb{F} = \C$.

The resulting necessary condition for $F$ to be a direct sum is the following:
\begin{prop}\label{prop: smith stong}
If $F$ has a linear factor then either $F$ is not a direct sum,
or $F = x^d - y^d$ for some linear forms $x, y$.
\end{prop}
\begin{proof}
By changing coordinates if necessary, suppose $x_1$ divides $F$.
Let $W = x_1^\perp = \langle \alpha_2,\dotsc,\alpha_n \rangle \subset V^* = T_1$.
This is a codimension $1$ subspace whose $d$-th power is in $F^\perp$,
that is $w_1 \dotsm w_d \in F^\perp$ for every $w_1,\dotsc,w_d \in W$.
Indeed, each monomial appearing in $F$ has at least one factor $x_1$ and hence is annihilated
by every product of $d$ elements of $W$.

Let $H = A_F = T/F^\perp$.
Then $W_1 = W/F^\perp_1$ is a codimension $1$ subspace of $H_1$
such that the product of any $d$ elements in $W_1$ is zero in $H$.
By Proposition~\ref{prop: smith and stong original}, either $H$ is indecomposable with respect to the connected sum operation,
or $H \cong \C[\alpha,\beta]/(\alpha\beta,\alpha^d-\beta^d) \cong (\C[\alpha]/\alpha^{d+1}) \# (\C[\beta]/\beta^{d+1})$.
In the first case it follows that $F$ is not a direct sum.
In the second case it follows that $F = x^d - y^d$ after a suitable change of coordinates.
\end{proof}

We remark in passing that this proposition is essentially just a restatement of Proposition~\ref{prop: smith and stong original}.
We have seen that if $F$ has a linear factor then there is a codimension $1$ subspace $W_1 \subset (A_F)_1$
such that the product of any $d$ elements in $W_1$ is zero in $A_F$, i.e., the $\times$-length of $W_1$ is less than $d$.
Conversely if $W_1$ is such a subspace, say $W_1 = x_1^\perp = \langle \alpha_2,\dotsc,\alpha_n \rangle$,
then every degree $d$ monomial in $\alpha_2,\dotsc,\alpha_n$ annihilates $F$,
so $F$ does not contain any terms that are monomials in just the variables $x_2,\dotsc,x_n$.
That is, every monomial appearing in $F$ has at least one factor $x_1$,
so $F$ is divisible by $x_1$.
Thus the hypotheses of Proposition~\ref{prop: smith and stong original}
   and Proposition~\ref{prop: smith stong} are equivalent.
Similarly, the conclusions are equivalent.

\section{Dimension of direct sum locus and uniqueness}\label{S: dimension of dirsum}
We discuss the uniqueness of the subspaces over which $F \in \DirSum$ splits
and we compute the dimension of $\DirSum$.

\subsection{Uniqueness of direct sum decompositions}

Thom conjectured in \cite{MR1067383} that every germ at $0$ of an analytic function $F$
has a unique finest decomposition as a sum of germs of functions in independent variables,
up to analytic equivalence.
This means that if 
\[
  F = F_1 + F_2 +\dotsb +F_k 
\]
   with $F_i$ in independent variables 
   and each $F_i$ cannot be written as such a sum, then Thom expected that for any other such decomposition 
   $F = G_1 + G_2 +\dotsb +G_l$, one must have $k=l$ and there exists an analytic isomorphism near $0$ preserving $F$ 
     and transporting $G_i$ to $F_i$ (up to permuting the $G_i$).
This was proved for quasi-homogeneous functions in \cite{MR582497}.
One may ask if a homogeneous polynomial has a unique finest decomposition
as a sum of polynomials in independent variables.

More generally, for a homogeneous polynomial $F$, we say that one direct decomposition is finer than another
if every direct summand subspace appearing in the second decomposition
is a direct sum of subspaces appearing in the first (finer) one.
That is, if $F = G_1 + \dotsb + G_k$ with $G_i \in S^d V_i$ for $i=1,\dotsc,k$ and $V_1 \oplus \dotsb \oplus V_k = V$
and also $F = G'_1 + \dotsb + G'_l$ with $G'_j \in S^d V'_j$ for $j=1,\dotsc,l$ and $V'_1 \oplus \dotsb \oplus V'_l = V$,
then the direct sum decomposition $G_1 + \dotsb + G_l$ is finer than $G'_1 + \dotsb + G'_k$
if every $V'_i$ is a direct sum of one or more of the $V_j$.

Clearly if $F$ is concise then a direct sum decomposition $F = G_1 + \dotsb + G_k$
is maximally fine if and only if each summand $G_i \in S^d V_i$ is concise with respect to $V_i$ and indecomposable as a direct sum.
And clearly every concise $F$ has a maximally fine direct sum decomposition.
The uniqueness question asks whether every concise $F$ has a unique maximally fine direct sum decomposition.

In fact, quadrics decompose as direct sums over many splittings of the vector space:
for example $x^2 + y^2 = (c x + s y)^2 + (s x - c y)^2$ for any $c,s$ such that $c^2+s^2=1$.
For this reason we usually restrict to degrees $d \geq 3$, and sometimes $d \geq 4$.
In these degrees the question of uniqueness has a positive answer.
\begin{thm}[{\cite[Thm.~3.7]{KleppePhD2005}}]\label{thm: unique finext direct sum decomposition}
Let $F$ be a concise form of degree $d \geq 3$.
Then $F$ has a unique maximally fine direct sum decomposition.
\end{thm}
In fact \cite[Thm.~3.7]{KleppePhD2005} holds in any characteristic
and gives a description of the subspaces appearing in the maximally fine direct sum decomposition.
Moreover, \cite[Prop.~3.1]{MR2680196}
  provides an analogous uniqueness decomposition for connected sums of Gorenstein Artin algebras.
However the proof of Theorem~\ref{thm: unique finext direct sum decomposition}
   requires some preparation which lies outside the scope of this paper.

Here we show a weaker statement: essentially that
the direct sum decomposition is uniquely determined
for forms in an open dense subset of $\DirSum$. 
This is sufficient for our purposes and does not require many tools
  other than those already introduced.

It is easy for binary forms.
\begin{prop}
Every direct sum in two variables of degree $d \geq 3$ has a uniquely determined decomposition.
\end{prop}
\begin{proof}
There is a unique (up to scalar) generator of the apolar ideal in degree $2$ (and another in degree $d$).
Writing $F = x^d - y^d$, this quadratic apolar generator is $Q = \alpha \beta$,
and the pair of subspaces $\langle x \rangle$, $\langle y \rangle$ over which $F$ decomposes
is determined as the pair of lines corresponding to the pair of points in projective space $\{ [x], [y] \} = V(Q)$.
\end{proof}

To go further we use the notion of \defining{compressed} algebras,
see Definition 3.11 and Proposition 3.12 of \cite{MR1735271}.
We recall, not the most general definition, but just the definition
in the case that $A = A_F = T/F^\perp$ is a graded Gorenstein Artin algebra of socle degree $d$.
In this case $A$ is compressed if, for each $i=0,\dotsc,d$, we have $\dim A_i = \min(\dim S^i(A_1), \dim S^{d-i}(A_1))$.
If we have chosen $T$ and the isomorphism $A = A_F = T/F^\perp$ in such a way that $F$ is concise,
 then $A$ is compressed if and only if  $\dim A_i = \min(\dim T_i,\dim T_{d-i})$.
When $F \in S^d V$ is general, $A_F$ is compressed \cite[Proposition~3.12]{MR1735271}.
(Recall that a general element of a variety is any element of a suitable dense open subset of the variety.)

\begin{lemma}\label{lemma: general indecomposable}
Let $d \geq 4$ and $n = \dim V \geq 2$. 
For $F \in S^d V$ general, $F^{\perp}$ has no quadratic generators and $F$ is not decomposable as a direct sum.
\end{lemma}
\begin{proof}
We have that $A_F$ is compressed.
This implies $F^\perp$ has no generators in degrees less than or equal to $d/2$,
   as $\dim (F^\perp)_i = \dim T_i - \dim A_i = 0$ for $0 \leq i \leq d/2$.
In particular $(F^\perp)_2 = 0$ so $F$ is not decomposable as a direct sum 
   (see Corollary~\ref{cor_quadratic_generators_of_a_direct_sum}).
\end{proof}

Table~\ref{table:plane cubics} shows that a general cubic in $n=3$ variables is not decomposable as a direct sum.
But a general binary cubic is decomposable as a direct sum:
Let $F = F(x,y)$ be a binary cubic with distinct roots.
By a linear substitution we may move those roots to be the cubic roots of unity;
in these coordinates $F = x^3-y^3$.

\begin{prop}\label{prop: general dirsum unique decomposition}
For $d \geq 4$ and $n \geq 2$, there is a dense subset of $\DirSum$
(that is, a union of a dense subset of each irreducible component of $\DirSum$)
such that for $F$ in this subset, $F$ decomposes as a direct sum over a uniquely determined pair of subspaces.
\end{prop}
\begin{proof}
Let $F \in \DirSum$ be arbitrary, $F = G + H$, with $G \in S^d V_1$, $H \in S^d V_2$, $V_1 \oplus V_2 = V$.
Now, let $G' \in S^d V_1$, $H' \in S^d V_2$ be general, and $F' = G' + H'$.
As $G' \to G$ and $H' \to H$ we have $F' \to F$.
Clearly $F'$ decomposes as a direct sum over $V_1 \oplus V_2$;
we claim that this is the unique pair of subspaces over which $F'$ decomposes.
We have $({F'}^\perp)_2 = ({G'}^\perp)_2 \cap ({H'}^\perp)_2$, see Lemma~\ref{lem: apolar of direct sum}.
Since $A_{G'}$ and $A_{H'}$ are compressed,
${G'}^\perp$ and ${H'}^\perp$ have no quadratic generators other than $({G'}^\perp)_2 = V_2^* V^*$ 
  and $({H'}^\perp)_2 = V^* V_1^*$.
In particular, then, $({F'}^\perp)_2 = V_1^* V^* \cap V^* V_2^* = V_1^* V_2^*$
and $V_1 \cup V_2$ is the zero locus $ V(({F'}^\perp)_2)$.
Hence $V_1$ and $V_2$ are uniquely determined by $F'$, as claimed.
This shows that there is a dense subset of $\DirSum$
whose elements decompose as direct sums over uniquely determined pairs of subspaces.
\end{proof}

On the other hand, there exists an open dense subset of cubic direct sums in three variables 
for which there is not a unique decomposition as a direct sum over two subspaces.
Indeed, a cubic  direct sum in three variables can be written as $F = x^3 + G(y,z)$.
If $G$ is a general cubic binary form, then (with another change of coordinates) $F = x^3 + y^3 + z^3$.
However, we do see that $F$ decomposes as a direct sum over
$V = \langle x \rangle \oplus \langle y \rangle \oplus \langle z \rangle$,
and this finest decomposition is uniquely determined by $F$, as
$\{ [x], [y], [z] \} = V(\alpha \beta, \alpha \gamma, \beta \gamma) = V((F^\perp)_2)$.

\subsection{Dimension}

For $V_1 \oplus V_2 = V$ let $\DirSum^*(V_1,V_2) = S^d(V_1) \oplus S^d(V_2)$;
note that this contains degenerate sums involving $0 \in S^d(V_1)$ or $0 \in S^d(V_2)$.
For $a+b=n$, $a \leq b$,
let $\DirSum^*(a,b)$ be the union of the $\DirSum^*(V_1,V_2)$ for $\dim V_1 = a$, $\dim V_2 = b$.
Each $\DirSum^*(a,b)$ is irreducible as it is the image of the natural projection\
\[
  \begin{split}
     S^d(\C^a) \times S^d(\C^b) \times \operatorname{GL}_n(\C) &\to S^d(V),\\
     (G, H, M) &\mapsto G(m_1,\dotsc,m_a) + H(m_{a+1},\dotsc,m_n)
  \end{split}
\]
where the $m_i$ are the columns of the matrix $M$.
Of course this is not injective.

For each $a+b=n$ let $\DirSum(a,b) \subset \DirSum^*(a,b)$ be the subset of $F$ which are indeed decomposable as direct sums
in which one term involves $a$ variables and the other involves $b$ variables, i.e., discarding those elements of
$\DirSum^*(a,b)$ in which one or both terms are identically zero.
Further, let $\DirSum^{\circ}(a,b) \subset \DirSum(a,b)$ be the subset of concise forms $F$.
Then $\DirSum^{\circ}(a,b)$ is a Zariski open subset of $\DirSum^*(a,b)$,
since its complement is defined by rank conditions on the catalecticant $C^1_F$.

Now $\DirSum = \bigcup_{a+b=n} \DirSum(a,b)$.
We see that $\DirSum$ contains the dense subset $\DirSum^{\circ} = \bigcup_{a+b=n} \DirSum^{\circ}(a,b)$,
i.e., a union of a dense open subset of each $\DirSum(a,b)$.

\begin{prop}\label{prop: dimension of dirsum}
For $d \geq 4$ and $n \geq 3$,
$\dim \DirSum^*(a,b) = 2ab + \binom{d+a-1}{a-1} + \binom{d+b-1}{b-1}$
and $\dim \DirSum = 2(n-1) + 1 + \binom{d+n-2}{n-2}$.
\end{prop}
\begin{proof}
Let $\DirSum^{\circ \circ}(a,b) \subset \DirSum^{\circ}(a,b)$ be the set of $F = G+H$
such that $A_G$ and $A_H$ are compressed.
There is a map from $\DirSum^{\circ \circ}(a,b)$ to $G(a,V) \times G(b,V)$
whose general fiber has dimension $\binom{d+a-1}{a-1} + \binom{d+b-1}{b-1}$.
This shows $\dim \DirSum^*(a,b)$ is as claimed.
This dimension is maximized when $(a,b) = (1,n-1)$.
\end{proof}

A more refined dimension formula is found in \cite[Thm.~3.47]{KleppePhD2005}.
Moreover an analogous formula for connected sum Gorenstein algebras is in \cite[Prop.~4.4]{MR2738376}.

\section{Apolar generators and limits of direct sums}\label{sect: apolar generators and limits of direct sums}

Let $F\in S^d V$ be a homogeneous polynomial of degree $d$. 
Recall that an equipotent generator of $F$ is a minimal generator
   of the ideal $F^{\perp}$ of degree $d$. 
In this section we collect results that relate quadratic generators
   to direct sums and to equipotent apolar generators.
Then we relate equipotent apolar generators
   to limits of direct sums.

\subsection{Quadratic generators}
Forms with an equipotent apolar generator have similar characteristics to forms which are direct sums.
Perhaps the best illustration of this is the behavior of quadratic apolar generators.

We make first the following easy observation:
\begin{prop}\label{prop: direct sum quadratic generators}
If $F$ is a concise direct sum in $n$ variables
then $F$ has at least $n-1$ quadratic apolar generators.
\end{prop}
It was previously shown by Meyer and Smith
that $F$ has at least one quadratic apolar generator
\cite[Lem.~VI.2.1]{MR2177162},
without assuming $F$ to be concise.
Moreover, \cite[Thm.~3.35]{KleppePhD2005}
  provides a calculation of all graded Betti numbers of $F^{\perp}$.

\begin{proof}
Say $F \in S^d V$ is a direct sum over $V = V_1 \oplus V_2$
with $\dim V_i = v_i$, $v_1 + v_2 = \dim V = n$.
Then by Corollary~\ref{cor_quadratic_generators_of_a_direct_sum} we have $V_1^* V_2^* \subset (F^\perp)_2$, a subspace of dimension $v_1 v_2 \geq n-1$.
By hypothesis there are no linear forms in $F^\perp$ so everything in $(F^\perp)_2$ is a minimal generator.
\end{proof}

Conversely, if $V = V_1 \oplus V_2$ and $V_1^{*} V_2^{*} \subset F^\perp$
then $F = F_1 + F_2$ where $F_1 \in S^d V_1$, $F_2 \in S^d V_2$, 
    see Corollary~\ref{cor_quadratic_generators_of_a_direct_sum}.
If furthermore $F$ is concise then $F_1, F_2 \neq 0$ and $F$ is a direct sum.

More generally, if $V = V_1 \oplus \dotsb \oplus V_s$,
then $F = F_1 + \dotsb + F_s$ where $F_i \in S^d V_i$
if and only if $\bigoplus_{i < j} V_i^* V_j^* \subset F^\perp$ as quadratic generators,
where $V_i^* = \bigcap_{j \neq i} V_j^\perp$.
(In coordinates, if each $V_i$ has a basis $x_{i,1},\dotsc,x_{i,n_i}$
then $V_i^*$ is spanned by the dual basis elements $\alpha_{i,1},\dotsc,\alpha_{i,n_i}$.)
If this holds and furthermore $F$ is concise then each $F_i \neq 0$ and $F$ is a direct sum of $s$ terms.

\begin{cor}
If $F$ is a concise form in $n$ variables
which is a direct sum of $s \geq 2$ terms
then $F^\perp$ has at least $(s-1)(2n-s)/2$ quadratic generators.
\end{cor}
\begin{proof}
When $F$ is a direct sum over $V = V_1 \oplus \dotsb \oplus V_s$,
$F^\perp$ contains $\bigoplus_{i < j} V_i^* V_j^*$ as quadratic generators,
see Corollary~\ref{cor_quadratic_generators_of_a_direct_sum}.
The fewest quadratic generators arise when the summands $V_1,\dotsc,V_s$
have dimensions $1,\dotsc,1,n+1-s$, yielding the statement.
\end{proof}

Less obviously we have
\begin{prop}\label{prop: degree d generator quadratic generators}
If $F$ is a concise form in $n$ variables
and $F$ has an equipotent apolar generator
then $F$ has at least $n-1$ quadratic apolar generators.
\end{prop}

\begin{proof}
Since $F$ is concise, $F^{\perp}$ has no linear generators.
Then the quadratic elements provided by Lemma~\ref{lem: beta gt 0 then minors} are minimal generators,
and there are at least $n-1$ independent ones, for example the $2 \times 2$ minors 
given by the first and $i$-th columns of \eqref{equ: n times 2 matrix with linear forms} for \mbox{$2 \leq i \leq n$}.
\end{proof}

Note that the Fermat hypersurface $x_1^d + \dotsb + x_n^d$ has $\binom{n}{2}$ quadratic apolar generators.
This is the maximum number possible for smooth forms, as the following easy observation shows.

\begin{prop}
If $F$ defines a smooth hypersurface of degree $d \geq 3$ on $\PP^{n-1}$
then $F^\perp$ has at most $\binom{n}{2}$ quadratic generators.
More generally if the set of points in $\PP^{n-1}$ at which
$F$ vanishes to order $\geq a$ has dimension $k$,
then $\dim (F^\perp)_{d-a+1} \leq \binom{n+d-a}{d-a+1} - n + k +1$.
\end{prop}
\begin{proof}
Otherwise $\PP F^\perp_{d-a+1} \subset \PP T_{d-a+1}$ necessarily has a $(k+1)$-dimensional intersection
with the Veronese variety $\veronese_{d-a+1}(\PP T_1)$, since it has codimension $\binom{n+d-a}{d-a+1}-n$.
For each $[\alpha^{d-a+1}]$ in this intersection,
$F$ vanishes to order at least $a$ at $[\alpha]$
by Lemma \ref{lem: apolarity singularity}.
This gives a $(k+1)$-dimensional set along which $F$ vanishes to order $a$.
The first statement follows with $a=d-1$ and $k=-1$ when $V(F)$ is smooth.
\end{proof}

Having the maximum number of quadratic apolar generators does not characterize Fermat hypersurfaces, however;
the concise plane cubics all have the maximum number of quadratic apolar generators, see Table~\ref{table:plane cubics}.

\subsection{Equipotent apolar generators and limits of direct sums}\label{sect: maximal degree apolar gens and limits}

Fix a degree $d$ and number of variables $n$.
Let $V$ be a vector space with $\dim V = n$.
In this section we prove first
that if $F \in S^d V$ has an equipotent apolar generator then $F$ is a limit of direct sums.
We next prove that if $F$ is also concise then the converse holds.
This assumption is needed, by Example~\ref{example: nonconcise limit of dirsum}.

\begin{thm}\label{thm: apoequ => limit of dirsum}
If $F$ has an equipotent apolar generator then $F$ is a limit of direct sums.
Moreover, either $F$ is a direct sum or it can be written in the following normal form,
  for some choice of basis $x_1,\dotsc,x_k$, $y_1,\dotsc,y_k$, $z_1,\dotsc,z_{n-2k}$ of $V$:
  \[
     F(x,y,z) = \sum x_i \frac{\partial H(y)}{\partial y_i} + G(y,z).
  \]
  Here $G \in S^d\langle\fromto{y_1}{y_k},\fromto{z_1}{z_{n-2k}} \rangle$ and $H \in S^d\langle\fromto{y_1}{y_k}\rangle$.
\end{thm}
\begin{proof}
We immediately reduce to the case that $F$ is concise:
If $F$ is concise over $W \subset V$, we will write $F$ as a limit of direct sums
which are in $S^d W$.

We assume $F^\perp$ has a generator in degree $d = \deg F$.
By Gorenstein symmetry \eqref{equ: Gorenstein symmetry for maximal degree apo gens}, 
   $\beta_{n-1, n}(F^{\perp})>0$.
By Lemma~\ref{lem: beta gt 0 then minors} there are linearly independent linear forms
$\ell_1,\dotsc,\ell_k$ for some $0 < k < n$ such that $F^\perp$ contains the $2 \times 2$ minors
of the matrix
\[
  \begin{pmatrix}
    \alpha_1 & \cdots & \alpha_k & \alpha_{k+1} & \cdots & \alpha_n \\
    \ell_1 & \cdots & \ell_k & 0 & \cdots & 0
  \end{pmatrix} .
\]
Let $L : T_1 \to T_1$ be the linear map given by $L(\alpha_i) = \ell_i$ for $1 \leq i \leq k$,
$L(\alpha_i) = 0$ for $i > k$.
That is, for all $i, j$, $\alpha_i L(\alpha_j) - \alpha_j L(\alpha_i) \in F^\perp$.
By linearity, $v L(w) - w L(v) \in F^\perp$ for all $v, w \in T_1$.
Let $\tilde{L} : \Wedge{2} T_1 \to F^\perp$
be defined by $\tilde{L}(v \wedge w) = v L(w) - w L(v) \in F^\perp$.

Since $0 < k < n$, $L$ is not zero or a scalar multiple of the identity and has a nontrivial kernel.
We begin by changing basis in $V^*$ (and dually in $V$) to put $L$ into Jordan normal form.
It turns out that if $L$ has distinct eigenvalues then $F$ decomposes as a direct sum
over the generalized eigenspaces of $L$; otherwise, if $L$ is a nonzero nilpotent matrix,
   then $F$ is a limit of direct sums.

Suppose first that $\lambda_i \neq \lambda_j$ are distinct eigenvalues of $L$.
Then there are some positive integers $\nu_i$, $\nu_j$ such that
$(L-\lambda_i)^{\nu_i} \alpha_i = (L-\lambda_j)^{\nu_j} \alpha_j = 0$
but $(L-\lambda_i)^{\nu_i-1} \alpha_i, (L-\lambda_j)^{\nu_j-1} \alpha_j \neq 0$.
We show that $\alpha_i \alpha_j \in \image(\tilde{L}) \subset F^\perp$, by induction on $\nu_i+\nu_j$.
The induction begins with $\nu_i=\nu_j=1$.
Then $\tilde{L}(\alpha_i \wedge \alpha_j) = \alpha_i L(\alpha_j) - \alpha_j L(\alpha_i) = (\lambda_j-\lambda_i) \alpha_i \alpha_j$.
If, say, $\nu_i > 1$, so $L(\alpha_i) = \lambda_i \alpha_i + \alpha_{i-1}$,
then $\tilde{L}(\alpha_i \wedge \alpha_j) = (\lambda_j - \lambda_i)\alpha_i \alpha_j + \alpha_{i-1}\alpha_j \in F^\perp$.
Since $\alpha_{i-1}\alpha_j \in F^\perp$ by induction, $\alpha_i \alpha_j \in F^\perp$.

This shows that, for the generalized eigenspace decomposition $V^* = \bigoplus_{\lambda} V^*_\lambda$,
we have $(V^*_\lambda)(\bigoplus_{\mu \neq \lambda} V^*_\mu) \subset F^\perp$ for each eigenvalue $\lambda$.
Thus $F = \sum_\lambda F_\lambda$, $F_\lambda \in S^d V_\lambda$ where
$V_\lambda = (V^*_\lambda)^* = \bigcap_{\mu \neq \lambda} (V^*_\mu)^\perp$ 
   by Corollary~\ref{cor_quadratic_generators_of_a_direct_sum}.
Since $F$ is concise each $F_\lambda$ must be nonzero (in fact, concise with respect to $V_\lambda$).
In this case, then, $F$ is a direct sum.

Now we see what happens when $L$ has just one eigenvalue.
Then $L$ is nilpotent, since $\ker L \ne 0$. 
We claim $\image(\widetilde{L^\nu}) \subset F^\perp$ for all $\nu \geq 1$.
Indeed,
\[
\begin{split}
  \widetilde{L^\nu} (\alpha \wedge \beta) &= \alpha \cdot L^\nu(\beta) - \beta \cdot L^\nu(\alpha) \\
    &= \alpha L(L^{\nu-1} \beta) - (L^{\nu-1} \beta)(L \alpha) + (L \alpha)(L^{\nu-1} \beta) - \beta (L^{\nu-1} L \alpha) \\
    &= \tilde{L}(\alpha \wedge L^{\nu-1} \beta) + \widetilde{L^{\nu-1}}(L(\alpha) \wedge \beta)
\end{split}
\]
which is in $F^\perp$ by induction.
Now suppose $L^{\nu+1} = 0 \neq L^{\nu}$; we replace $L$ with $L^{\nu}$, so we can assume $L^2 = 0 \neq L$.
Say $k = \rank L$.
Then the Jordan normal form of $L$ yields a basis $\alpha_1,\dotsc,\alpha_k$, $\beta_1,\dotsc,\beta_k$,
$\gamma_1,\dotsc,\gamma_{n-2k}$ of $V^*$ such that $L(\beta_i) = \alpha_i$ and $L(\alpha_i) = L(\gamma_i) = 0$.
So $L$ can be written in block form with respect to this basis,
\[
  L =
  \begin{blockarray}{cccc}
  {\scriptstyle \alpha} & {\scriptstyle \beta} & {\scriptstyle \gamma} & ~ \\
  \begin{block}{(ccc)c}
    0 & I & 0 & {\scriptstyle \alpha} \\
    0 & 0 & 0 & {\scriptstyle \beta} \\
    0 & 0 & 0 & {\scriptstyle \gamma} \\
  \end{block}
  \end{blockarray}.
\]
We give $V$ the dual basis $x_1,\dotsc,x_k$, $y_1,\dotsc,y_k$, $z_1,\dotsc,z_{n-2k}$.
Now $\tilde{L}(\alpha_i \wedge \beta_j) = \alpha_i \alpha_j \in F^\perp$,
and $\tilde{L}(\gamma_i \wedge \beta_j) = \gamma_i \alpha_j \in F^\perp$.
Since
\[
  \langle \alpha_1,\dotsc,\alpha_k \rangle \langle \alpha_1,\dotsc,\alpha_k, \gamma_1,\dotsc,\gamma_{n-2k} \rangle \subset F^\perp ,
\]
we have
$F = \sum x_i H_i(y) + G(y,z)$
where $\deg H_i = d-1$, $\deg G = d$.
Furthermore $\tilde{L}(\beta_i \wedge \beta_j) = \alpha_j \beta_i - \alpha_i \beta_j \in F^\perp$,
so $\partial H_i / \partial y_j = \partial H_j / \partial y_i$.
Thus there exists $H(y)$ such that $H_i = \partial H / \partial y_i$.
This shows the normal form part of the statement of the theorem.
Finally we write $F$ as a limit of direct sums as follows:
\begin{equation}\label{eq: limit of dirsum}
  F = \lim_{t \to 0} \frac{1}{t} \Big( H(y_1 + t x_1, \dotsc, y_k + t x_k) + t G(y,z) - H(y) \Big),
\end{equation}
which for $t \neq 0$ is a direct sum over $\langle y_i + t x_i \rangle \oplus \langle y,z \rangle$.
\end{proof}  

\begin{example}
Let $F = xy^{d-1}$ so that $F^\perp = \langle \alpha^2, \beta^d \rangle$.
Then $F^\perp$ contains the $2 \times 2$ minors of the matrix
\[
  \begin{pmatrix}
    \alpha & \beta \\
    0 & \alpha
  \end{pmatrix}.
\]
In the notation of the above proof, $L : T_1 \to T_1$ is given by $L(\alpha) = 0$, $L(\beta) = \alpha$.
Then $L$ is nilpotent, $L^2 = 0$.
The next step in the proof provides a decomposition $F = x H_1(y)$, where $\deg H_1 = d-1$; so $H_1(y) = y^{d-1}$.
We have $H_1 = \partial H/\partial y$ where $H(y) = (1/d) y^d$.
In the proof's notation, $G=0$.
Of course then
\[
  F = xy^{d-1} = \lim_{t \to 0} \frac{1}{dt} \left( (y+tx)^d - y^d \right) .
\]
\end{example}

\begin{example}
Let $F = x^2 y - y^2 z$.
We saw in Example~\ref{example: plane cubic apoequ not dirsum} that
$F^\perp = \langle \gamma^2, \alpha \gamma, \alpha^2 + \beta \gamma, \beta^3, \alpha \beta^2 \rangle$,
so $F$ has two equipotent apolar generators.
And $F^\perp$ contains the $2 \times 2$ minors of the matrix
\[
  \begin{pmatrix}
    \alpha & \beta & \gamma \\
    -\gamma & \alpha& 0
  \end{pmatrix}.
\]
In the notation of the above proof, the endomorphism corresponding to this matrix
is $L : T_1 \to T_1$, given by $L(\alpha) = -\gamma$, $L(\beta) = \alpha$, $L(\gamma) = 0$.
Note that $L$ is nilpotent, $L^3 = 0$.
We replace $L$ with $L' = L^2$, represented by the matrix
\[
  \begin{pmatrix}
    \alpha & \beta & \gamma \\
    0  &  -\gamma  &  0
  \end{pmatrix}.
\]
Again $L'$ is nilpotent, $L'^2 = 0$.
Although the labeling of variables is different than in the proof, the proof's next step
yields $F = z H_1(y) + G(x,y)$ where $H_1(y) = -y^2$; apparently $G(x,y) = x^2 y$.
Then $H(y) = -(1/3) y^3$.
We get
\[
\begin{split}
  F &= \lim_{t \to 0} \frac{1}{t} \left( H(y+tz) + tG(x,y) - H(y) \right) \\
    &= \lim_{t \to 0} \frac{1}{3t} \left( -(y+tz)^3 + t x^2 y + y^3 \right) ,
\end{split}
\]
a limit of direct sums over the subspaces $\langle y+tz \rangle \oplus \langle x,y \rangle$ for $t \neq 0$.
\end{example}

Theorem~\ref{thm: apoequ => limit of dirsum} has several parallels in \cite{KleppePhD2005}.
The linear map $L$ in its proof corresponds to one of the matrices in \cite[Def.~2.14]{KleppePhD2005}.
The nilpotent case is covered by \cite[Thm.~4.5]{KleppePhD2005}.
In the case of distinct eigenvalues, the projections onto the distinct eigenspaces
give the orthogonal idempotent matrices discussed in \cite[Prop.~3.5]{KleppePhD2005}.
An extended result in this case is given in \cite[Thm.~3.7]{KleppePhD2005}.

Now we prove the converse (c.f. \cite[Lem.~4.2]{KleppePhD2005}).

\begin{thm}\label{thm: limit of s fold direct sum then s-1 max deg apo gens}
If $F$ is a concise limit of $s$-fold direct sums then $F$ has at least $s-1$ equipotent apolar generators.
\end{thm}
\begin{proof}
Suppose $F=F_0$ is a concise limit of $s$-fold direct sums, $F_0 = \lim F_t$.
Let $J$ be the flat limit of the ideals $F_t^\perp$.
We have $J \subset F^{\perp}$, since differentiation varies continuously
as the $F_t$ varies regularly.
Indeed, for $\Theta \in J$, $\Theta = \lim \Theta_t$ for $\Theta_t \in F_t^\perp$,
so $\Theta_t \aa F_t = 0$ for $t \neq 0$;
hence $\Theta \aa F = \lim \Theta_t \aa F_t = 0$, so $\Theta \in F^\perp$.
By Proposition~\ref{prop: semicontinuity of graded Betti numbers}, upper-semicontinuity of graded Betti numbers,
$\beta_{n-1,n}(J) \geq s-1$.

Now there is no general inequality between the graded Betti numbers of
an arbitrary homogeneous ideal $I$ and homogeneous subideal $J \subset I$;
$\beta_{i,j}(I) > \beta_{i,j}(J)$, $\beta_{i,j}(I) < \beta_{i,j}(J)$, and $\beta_{i,j}(I) = \beta_{i,j}(J)$
all are possible.
However in this simple case we do have the inequality we are looking for by Lemma \ref{lem: beta inequality subideal}.
That is,
$\beta_{n-1,n}(F^{\perp}) \ge  \beta_{n-1,n}(J) \geq s-1$,
since $F$ is concise, meaning $(F^{\perp})_1=0$.
By Gorenstein symmetry \eqref{equ: Gorenstein symmetry for maximal degree apo gens}
   there are at least $s-1$ minimal generators of degree $d$ in $F^{\perp}$.
\end{proof}

This completes the proof of Theorem~\ref{thm: apoequ = closure of dirsum},
which comprises Theorems~\ref{thm: apoequ => limit of dirsum} and \ref{thm: limit of s fold direct sum then s-1 max deg apo gens}.

\begin{example}\label{example: form of plane direct sums}
Let $F$ be a concise plane curve ($n=3$) of degree $d$ having an equipotent apolar generator.
Either $F$ is a direct sum, $F = x^d + G(y,z)$, or else $F$ is a limit of direct sums
of the form $F = x y^{d-1} + G(y,z)$ by the normal form part of Theorem~\ref{thm: apoequ => limit of dirsum}.
Note that if $G(y,z)$ includes terms $a y^{d-1} z + b y^d$ then replacing $x$ with
$x + az + by$ gives us
\[
  F = x y^{d-1} + z^2 G_{d-2}(y,z)
\]
where $\deg G_{d-2} = d-2$.
Conversely if $F$ is of this form then $F$ is a limit of direct sums
and has an equipotent apolar generator.

Thus a concise plane curve $F$ is a limit of direct sums and has an equipotent apolar generator
if and only if, after a linear change of coordinates,
either $F = x^d + G(y,z)$ or $F = x y^{d-1} + z^2 G_{d-2}(y,z)$.
\end{example}

\subsection{Lower bound for degree of apolar generators}\label{sect: lower bound on degree}

Here we prove Theorem~\ref{thm: apolar generator degree lower bound},
a lower bound for the maximum degree of the apolar generators
of a form $F$ in terms of the degree $d$ of $F$ and the number $n$ of variables:
we show that $F^\perp$ always has a minimal generator of degree at least $(d+n)/n$.

\begin{lemma}\label{lemma: socle degree of complete intersection}
Suppose $I \subset T$ is a homogeneous complete intersection ideal of codimension $n$.
Then $I = G^\perp$ for some $G \in S$.
If the minimal generators of $I$ are homogeneous of degrees $\fromto{\delta_1}{\delta_n}$,
then $\deg G = \delta_1 + \dots + \delta_n - n$.
\end{lemma}
\begin{proof}
The existence of $G$ follows from Theorem 21.6 of \cite{eisenbud:comm-alg}.
The degree of $G$ is equal to the socle degree of $A_G \cong T/I$,
which is $\sum_{i=1}^n (\delta_i - 1)$ by, for example, Exercise 21.16 of \cite{eisenbud:comm-alg}.
\end{proof}

We will deduce Theorem~\ref{thm: apolar generator degree lower bound}
from the following slightly stronger proposition.

\begin{prop}\label{prop: apolar generator degree lower bound strengthening}
Let $F$ be a homogeneous form of degree $d$ in $n$ variables.
Suppose $F^\perp$ has minimal generators $\Theta_1,\dotsc,\Theta_s$ such that
$\deg \Theta_i = d_i$ for each $i$, and $d_1 \leq \dotsb \leq d_s$.
Let $\delta$ be an integer such that the ideal $(F^\perp)_{\leq \delta}$ is
$\gotm$-primary.
Assume $d_k = \delta < d_{k+1}$ or $k = s$ and $\delta = d_s$; necessarily $k \geq n$.
Then $d \leq d_k + d_{k-1} + \dotsb + d_{k-n+1} - n$.
\end{prop}

\begin{proof}
For each $i = k, k-1, \dotsc, k-n+1$, let $\Psi_i \in (F^\perp)_{d_i}$ be general.
Since $(F^\perp)_{\leq d_k}$ is $\gotm$-primary,
$(F^\perp)_{d_i}$ is a basepoint free linear series on $V(\Psi_k,\dotsc,\Psi_{i+1})$
for each $i$.
Then by Bertini's Theorem \cite[Thm.~I.6.3]{MR725671} and downward induction on $i$,
the ideal $(\Psi_k,\Psi_{k-1},\dotsc,\Psi_{i})$ is a complete intersection for each $i \geq k-n+1$.
In particular $I = \langle \Psi_k,\dotsc,\Psi_{k-n+1} \rangle$ is a complete intersection of codimension $n$.
By Lemma~\ref{lemma: socle degree of complete intersection}, $I=G^{\perp}$
for a form $G$ of degree $d_k+\dotsb+d_{k-n+1}-n$.
And $F = \Theta \aa G$ for some $\Theta \in T$ by Lemma~\ref{lemma: apolar containment},
so $\deg F \leq \deg G$.
\end{proof}

\begin{proof}[Proof of Theorem~\ref{thm: apolar generator degree lower bound}]
Let $d_1 \leq \dotsb \leq d_s$ be the degrees of the minimal generators of $F^\perp$,
as in the previous proposition.
Then regardless of the value of $k$ we have $d \leq d_k + \dotsb + d_{k-n+1} - n \leq n \delta - n$.
\end{proof}

Suppose $F$ is a concise homogeneous polynomial in $n$ variables for which $\delta =2$,
  i.e., $F^{\perp}$ is generated by quadrics. 
Then Theorem~\ref{thm: apolar generator degree lower bound} implies $d = \deg F \le n$.
Moreover, the proof of the theorem shows $F = \Theta \aa G$ for some $\Theta \in T$ and $G\in S$
such that $\deg G = n$ and $G^{\perp}$ is a complete intersection of $n$ quadrics.
For example, let $F$ be a determinant of a generic $k \times k$ matrix,
\[
  F= \det \begin{pmatrix}
            x_{11} & \cdots & x_{1k} \\ 
            \vdots &        & \vdots \\
            x_{k1} & \cdots & x_{kk} 
          \end{pmatrix}.
\]
In this case $\deg F =k$ and the number of variables is $k^2$.
Then as $G$ we may take the monomial $\prod_{i,j=1}^k x_{ij}$.

However it is not true that every homogeneous polynomial of the form $\Theta \aa G$
must have $(\Theta \aa G)^{\perp}$ generated by quadrics.
For example, let $G = x_1 \dotsm x_6$, let $\Theta = \alpha_1 \alpha_2 \alpha_3 - \alpha_4 \alpha_5 \alpha_6$,
and let $F = \Theta \aa G = x_4 x_5 x_6 - x_1 x_2 x_3$.
Then $G^\perp$ is a complete intersection of quadrics but $F^\perp$ has a minimal generator of degree $3$
by Theorem~\ref{thm: direct sum generator}, namely, $\alpha_4 \alpha_5 \alpha_6 + \alpha_1 \alpha_2 \alpha_3$.

The problem of classification of all homogeneous polynomials $F$ with $F^{\perp}$ generated by quadrics appears to be difficult.
We expect that the answer must be complex: if $F$ is the permanent of a generic symmetric matrix
then $F^\perp$ has minimal generators of degree $3$,
  while the apolar ideal of the determinant of the same matrix has only quadratic generators,
  see Theorems 3.23 and 3.11 of \cite{Shafiei:2013fk}.
However, the above discussion shows that it might be helpful to first classify $G$ with $d= \deg G = n$ and $G^{\perp}$ generated by quadrics.

Even the classification of $G$ is difficult.
For $d=n=2$, any rank two quadric $G$ has $G^{\perp}$ generated by quadrics.
For $d=n=3$, the plane cubic $G$ has $G^{\perp}$ generated by quadrics if and only if it is concise and has no degree $3$ minimal generators, equivalently, $G$ is not a limit of direct sums, 
  see Table~\ref{table:plane cubics}.
In particular, the general plane cubic has its apolar ideal generated by quadrics.
For $d=n\ge 4$, the general form produces a \emph{compressed algebra} 
  (see proof of Lemma~\ref{lemma: general indecomposable}), 
  and thus has no quadratic generators in the apolar ideal.

\section{Variation in families}\label{sect: variation in families}

If $F_t \to F$, it does not necessarily follow that $F_t^\perp \to F^\perp$ or $A_{F_t} \to A_F$
as flat families.
For example, $x^d + t y^d \to x^d$ as $t \to 0$,
but $(x^d)^\perp = \langle \alpha^{d+1},\beta \rangle$ is not the flat limit of the ideals
$(x^d + ty^d)^\perp = \langle t\alpha^d - \beta^d, \alpha\beta \rangle$.
This can also occur if all polynomials in the family are concise,
for example $x^d + y^d + t(x+y)^d \to x^d + y^d$ as $t \to 0$.

When $\{F_t\}$ is a family of polynomials such that $\{F_t^\perp\}$ is a flat family,
we say $\{F_t\}$ is an \defining{apolar family} and $F_t \to F_0$ is an \defining{apolar limit}.
It is equivalent to say $\{A_{F_t}\}$ is a flat family. 

Since we only consider homogeneous polynomials, $\{F_t\}$ is an apolar family if and only if the Hilbert functions of the $A_{F_t}$ are locally constant
\cite[Exer.~20.14]{eisenbud:comm-alg}.
When this holds, in particular their sum, the length of $A_{F_t}$, is locally constant.
On the other hand, the values of the Hilbert function are lower semicontinuous in $t$ since they are
the ranks of catalecticants, which are linear maps depending regularly on $t$.
Thus if $\length(A_{F_t})$ is constant in $t$ then the Hilbert function must also be constant.
We underline that this implication is dramatically false
     if we consider flat families of apolar algebras of non-homogeneous polynomials,
     see for instance \cite{MR713096} or \cite{casnati_jelisiejew_notari_Hilbert_schemes_via_ray_families}.

\begin{remark}
Families of homogeneous polynomials with constant Hilbert function are intensively studied.
If $T$ is a finite sequence of positive integers, then the set of all homogeneous polynomials of degreee $d$ 
  with Hilbert function $T$ is denoted in the literature by $Gor(T)$, see for instance \cite{MR1735271}.
In particular, a family $F_t$ is an apolar family if and only if for some $T$ we have $F_t \in Gor(T)$ for all $t$.
\end{remark}

\begin{prop}\label{prop: cubic apolar limit}
Every concise limit of cubic forms is an apolar limit.
\end{prop}
\begin{proof}
The Hilbert function of the apolar algebra of any concise cubic form
in $n$ variables is $1,n,n,1$,
so every concise cubic form has apolar length $2n+2$
and every family of concise cubic forms is automatically an apolar family.
\end{proof}

Proposition \ref{prop: cubic apolar limit} shows that when $d=3$,
$\ApoLim \cap \Con = \ApoEqu \cap \Con$.
We will show that for some $n$ and sufficiently large $d$,
$\ApoLim \cap \Con \subsetneqq \ApoEqu \cap \Con$.
Then will we show that for $n=3$ once again
$\ApoLim \cap \Con = \ApoEqu \cap \Con$.
First, we introduce cactus rank and use it to examine some cases
in which we can show that a form $F$ has numerous equipotent
apolar generators.

\subsection{Cactus rank and number of maximal degree apolar generators} \label{sect: cactus rank}

In this section we examine some cases in which we can show that a form $F$
not only has a maximal degree apolar generator, but in fact has several.
Throughout this section we assume $n = \dim V \geq 2$.

\begin{prop}\label{prop: rank n is direct sum}
Suppose $F$ is concise and in addition the Waring rank $r(F)=n$.
Then $F$ is a direct sum of $n$ terms
and $F$ has at least $n-1$ equipotent apolar generators.
Furthermore if $d>2$, $F$ has exactly $n-1$ equipotent apolar generators.
\end{prop}

\begin{proof}
    Up to a choice of coordinates, $F$ is the equation of the Fermat hypersurface.
\end{proof}

\begin{prop}\label{prop: border rank n is a limit of direct sums}
Suppose $F$ is concise and in addition the border rank $br(F)=n$.
   Then $F$ is a limit of direct sums of $n$ terms
     and $F$ has at least $n-1$ equipotent apolar generators.
\end{prop}

\begin{proof}
   $F$ is a limit of polynomials as in Proposition~\ref{prop: rank n is direct sum}.
   The second statement follows by Theorem~\ref{thm: limit of s fold direct sum then s-1 max deg apo gens}.
\end{proof}

There is another notion of rank of polynomials, namely the \defining{cactus rank} \cite{MR2842085}.
It is also called the \defining{scheme length} in \cite[Definition~5.1]{MR1735271}.
The cactus rank of $F \in S^d V$ is the minimal length of a zero dimensional subscheme
$R \subset \PP V$ such that $[F] \in \langle \veronese_d(R) \rangle$,
or equivalently $I(R) \subset F^{\perp}$ \cite[Prop.~3.4(vi)]{MR3121848}.
We prove an analogue of Propositions~\ref{prop: rank n is direct sum} and
\ref{prop: border rank n is a limit of direct sums}
for cactus rank:
\begin{thm}\label{thm: cactus rank n is a limit of direct sums}
Suppose $F$ is concise and in addition the cactus rank $cr(F)=n$.
     Then $F$ is a limit of direct sums
     and $F$ has at least $n-1$ equipotent apolar generators.
\end{thm}
Unlike in Propositions \ref{prop: rank n is direct sum} and \ref{prop: border rank n is a limit of direct sums},
we do not claim that $F$ is a limit of $n$-fold direct sums.

This theorem is proven in three steps.
The first step (Lemma~\ref{lemma: cactus rank n and d ge 2n}) 
   is the same statement, but with an extra assumption that $d \ge n+1$.
In the second step we use the first step to prove a property 
   about syzygies of zero dimensional schemes embedded in a concisely independent way, which might be of interest on its own.
In the final step we use the syzygies of schemes to prove the theorem.

To obtain a number of minimal generators in some degree we compare two ideals $J \subset I \subset T$ 
   (for example $I= F^{\perp}$), 
   where $J$ is generated by $I_{\le \delta}$.
Then we compare the Hilbert functions of $T/I$ and $T/J$.
The smallest integer $d$ where $h_{T/I}(d) \ne h_{T/J}(d)$ 
   is a degree in which there must be a minimal generator of $I$;
   in fact there are at least $h_{T/J}(d)-h_{T/I}(d)$ minimal generators of degree $d$.

\begin{lemma}\label{lemma: cactus rank n and d ge 2n}
      With $F$ as in Theorem~\ref{thm: cactus rank n is a limit of direct sums}, if in addition $d \ge n+1$, 
           then $h_{A_F}$, the Hilbert function of $A_F$, is $(1,n,n,\dotsc, n,1, 0, \dotsc)$
           and  $F^{\perp}$ has exactly $n-1$ minimal generators in degree $d$.
\end{lemma}

\begin{proof}
   Consider an ideal $I \subset T$ defining the scheme realizing the cactus rank of $F$.
   That is, $I$ is a saturated homogeneous ideal, $I \subset F^{\perp}$ and $B= T/I$ is a graded algebra with constant Hilbert polynomial equal to $cr(F)=n$.
   Let $I' = I_{\leq n}$ be the ideal generated by the forms in $I$ of degree
   less than or equal to $n$, and let $B' = T/I'$.

   First note that we have the following inequality of Hilbert functions: $h_{A_F} \le h_B \le n$.
   Since $F$ is concise, $n= h_{A_F}(1)=h_{A_F}(d-1)$.
   Thus $h_B(1) =n$, and since $h_B$ is nondecreasing \cite[Rem.~2.8]{MR2309930}, 
     we must have $h_B(i) =n$ for all $i \ge 1$.
   In particular $h_{B'}(n) = h_B(n) = n$ since $I'_n = I_n$,
   and $h_{B'}(n+1) \geq h_B(n+1) = n$.
   On the other hand, by Macaulay's Growth Theorem \cite[Thm~3.3]{MR1648665} or \cite[Cor.~5.1]{MR3121848}
   we have $h_{B'}(n+1) \leq h_{B'}(n)$.
   Thus $B'$ realizes the maximal possible growth of a Hilbert function from $h_{B'}(n) = h_{B'}(n+1)$ onwards
      and hence by Gotzmann's Persistence Theorem \cite[Thm~3.8]{MR1648665} or \cite[Cor.~5.3]{MR3121848}
      we have $h_{B'}(i) = n = h_B(i)$ for all $i \geq n$.
      Thus $I'_i = I_i$ for all $i \geq n$.
      This shows that the ideal $I$ is generated by $I_{\le n}$.
      
   By Macaulay's Growth Theorem
   we have $n = h_{A_F}(d-1) \leq h_{A_F}(d-2) \leq \dotsb \leq h_{A_F}(n) \leq n$.
   In particular, $I$ and the ideal generated by $I_n = (F^{\perp})_{n}$ agree
   in degrees $n, n+1, \dotsb, d-1$.
   However $h_{A_F}(d) = 1$ while $h_B(d) = n$.
   Thus $F^{\perp}$ needs exactly $n-1$ minimal generators in degree $d$.
   Moreover, $I$ is saturated and $F^{\perp}_{\le n}$ is a saturation of $(F^{\perp})_{n} = I_{n}$, hence:
   \[
     F^{\perp}_{\le n} = ((F^{\perp})_{n})^{sat} = (I_{n})^{sat} \subset I^{sat} = I \subset F^{\perp}.
   \]
   Therefore $F^{\perp}$ and $I$ agree up to degree $d-1$, and the Hilbert function of $A_F$ is $(1,n,n,\dots, n, 1)$.
\end{proof}

We will consider $R \subset \PP V$, a zero dimensional locally Gorenstein subscheme.
Such schemes arise naturally when considering cactus rank.
Namely it follows from \cite[Lem.~2.3]{MR3121848}
  that if $cr(F) =n$, 
  then there exists a zero dimensional locally Gorenstein subscheme $R$
  such that $\length R =n$ and $F \in \langle \veronese_d(R) \rangle$.
Here we will study such $R$ which are embedded into $\PP V$ in \emph{a concisely independent way},
  that is $\length R = \dim V$ and $R$ is not contained in any hyperplane.
Note that every finite scheme can be embedded in a concisely independent way:
By, for example, \cite[Lem.~2.3]{MR3092255}, if $R \subset \PP V$ is a finite scheme of length $r$,
then the Veronese re-embedding $\veronese_{r-1}(R) \subset \PP(S^{r-1} V)$
spans an $(r-1)$-dimensional projective subspace in which $R$ is embedded concisely independently.
  
\begin{example}\label{example_concise scheme}
   Suppose $G$ is a concise cubic in $6$ variables $\fromto{x_1}{x_6}$.
   Let $R = \Spec A_G$ be the zero-dimensional Gorenstein scheme of length $14$ determined 
     by $G$.
   We will now describe in some detail a concisely independent embedding of $R$ into $\PP^{13} = \PP V$,
     where $V = \langle \fromto{x_1}{x_6}, \fromto{y_1}{y_6}, z, w \rangle$,
     and $T= \C[\fromto{\alpha_1}{\alpha_6}, \fromto{\beta_1}{\beta_6}, \gamma, \delta ]$ is the coordinate ring.
     This embedding will play a role in Example~\ref{example_limit_polynomial_but_not_apolar_limit},
     our explicit example of a homogeneous form which is a limit of direct sums but not an apolar limit of direct sums.
     
   Consider the ideal $I$ generated by:
   \begin{enumerate}
    \item \label{item_generators_of_I__apolar_to_G}
          The apolar ideal of $G$ in $\C[\fromto{\alpha_1}{\alpha_6}]$.
          This provides $15$ quadric minimal generators and perhaps some cubics.
    \item \label{item_generators_of_I__alpha_beta}
           30 quadrics $\alpha_i\beta_j$ for $i \ne j$, $i,j \in \setfromto{1}{6}$.
    \item \label{item_generators_of_I__gamma_smthg}
           13 quadrics $\alpha_i \gamma$, $\beta_i \gamma$, $\gamma^2$ for $i \in  \setfromto{1}{6}$.
    \item \label{item_generators_of_I__beta_beta} 
           21 quadrics $\beta_i \beta_j$ for $i,j \in  \setfromto{1}{6}$.
    \item \label{item_generators_of_I__alpha_beta_minus_gamma_delta}
           6 quadrics $\alpha_i \beta_i - \gamma \delta$ for $i \in \setfromto{1}{6}$.
    \item \label{item_generators_of_I__Theta_minus_beta_delta}
           6 quadrics obtained in the following way:
           For $i \in \setfromto{1}{6}$
           let $\Theta_i \in \C[\fromto{\alpha_1}{\alpha_6}]$ be a quadric such that
           $\Theta_i \aa G = x_i$ (these quadrics exist, since $G$ is concise with respect to $\langle \fromto{x_1}{x_6}\rangle$).
          Then include in $I$ the following quadrics: $\Theta_i - \beta_i\delta$. 
   \end{enumerate}
   Altogether we obtain $91$ quadrics and perhaps some cubics (depending on $G$).
   The radical of the homogeneous ideal $I$ is generated by:
   \[
     \sqrt{I} = \langle \fromto{\alpha_1}{\alpha_6}, \fromto{\beta_1}{\beta_6}, \gamma \rangle.
   \]
   To see this, $\alpha_i^4 \in I$ by \ref{item_generators_of_I__apolar_to_G},
                $\beta_i^2 \in I$ by \ref{item_generators_of_I__beta_beta},
                $\gamma^2  \in I$ by \ref{item_generators_of_I__gamma_smthg}.
   Thus the projective scheme defined by $I$ is supported at the single point $[w] \in \PP V$,
     which is contained in the open subset $\delta \ne 0$.
   Evaluating the generators of $I$ at $\delta =1$,
     the reader can easily check that the scheme supported at $[w]$ is isomorphic to $R$,
     and also that there are no linear forms in $T_1$ that contain this scheme.
   
   Also it is not difficult to see that the Hilbert function of $T/I$ is $(1, 14, 14, 14, \dotsc)$.
   We combine this information
      with the fact that the Hilbert function of a saturated ideal is non-decreasing, see \cite[Rem.~2.8]{MR2309930}.
   We conclude that $I$ is the saturated ideal defining $R \subset \PP V$, 
      and the embedding of $R$ is concisely independent, because $I_1=0$ and $\length R = \dim V =14$.
   
   Finally, we remark that for general $G \in S^3 \C^6$, the scheme $R$ is the shortest non-smoothable Gorenstein scheme.
   See \cite[Lemma~6.21]{MR1735271}, where it is shown that $R$ is non-smoothable,
   and \cite{Casnati20141635}, where it is shown that all shorter Gorenstein schemes are smoothable.

\end{example}
  
\begin{prop}\label{prop: concisely independent then beta ge n-1}
    Suppose $R$ is a finite Gorenstein scheme as above,  $\length R =n$ and $R \subset \PP V$ is concisely independent.
    Let $J \subset T$ be the saturated homogeneous ideal of $R$.
    Then $\beta_{n-1,n}(J) \ge n-1$. 
\end{prop}

\begin{proof}
     Consider a general $F\in \langle \veronese_d (R) \rangle$ for some $d \ge 2n$.
     It follows from \cite[Lem.~2.3]{MR3121848}
        that $F$ is not contained in $\langle \veronese_d(Q) \rangle$ for any $Q \subsetneqq R$.
     It further follows from \cite[Cor.~2.7]{MR3092255} that $R$ is determined by $F$.
     Namely, $R$ is the unique subscheme of $\PP V$ of length $n$
        such that $F \in \langle \veronese_d (R) \rangle$.
     Also $cr(F) =n$.
     By \cite[Thm.~1.6]{MR3121848}, $J = (F^{\perp})_{\le n}$.
     In particular, $h_{A_F}(1) = h_{T/J}(1) =n$, that is $F$ is concise.
     
     By Lemma~\ref{lemma: cactus rank n and d ge 2n}
        and Gorenstein symmetry \eqref{equ: Gorenstein symmetry for maximal degree apo gens}
        we have $\beta_{n-1, n}(F^{\perp}) \ge n-1$.
     The syzygies involve only quadratic generators of $F^{\perp}$,
        so they also exist in $(F^{\perp})_{\le n} = J$,
        and $\beta_{n-1, n}(J) \ge n-1$ as claimed.
\end{proof}

\begin{proof}[Proof of Theorem~\ref{thm: cactus rank n is a limit of direct sums}]
     Let $R\subset \PP V$ be a locally Gorenstein scheme of length $n$ such that $F \in \langle \veronese_d(R) \rangle$, 
        whose existence is guaranteed by the definition of cactus rank and \cite[Lem.~2.3]{MR3121848}.
     Let $J \subset T$ be the homogeneous saturated ideal defining $R$.
     We have $J \subset F^{\perp}$, 
        and thus 
     \[
        \beta_{n-1, n}(F^{\perp}) \ge \beta_{n-1, n}(J) \ge n-1
     \]
     by Lemma \ref{lem: beta inequality subideal}
        and Proposition~\ref{prop: concisely independent then beta ge n-1}.
     By Gorenstein symmetry \eqref{equ: Gorenstein symmetry for maximal degree apo gens} 
        the ideal $F^{\perp}$ must have at least $n-1$ minimal generators of degree $d$, as claimed.
\end{proof}

\subsection{Cleavable and uncleavable schemes}\label{sect: non-apolar limit}

In this section we give examples of limits of direct sums
which cannot be obtained as apolar limits of direct sums;
these are points in $\ApoEqu$ that are not in $\ApoLim$.

Recall that we are working over the field $\C$, in particular in characteristic $0$.
Let $n_0$ be the minimal integer such that there exists
a non-smoothable locally Gorenstein scheme of length $n_0$.
As we will see, for our purposes we do not need to know the value of $n_0$, just that there is such a value.
And in fact the value of $n_0$ may depend on the characteristic.

Although the value of $n_0$ does not matter for us,
it turns out that in characteristic $0$ this value has been determined very recently.
It is well known that $12 \le n_0 \le 14$,
see \cite[Sect.~6, Sect.~8.1]{MR3121848} for an overview and references,
   and the recent work \cite{Casnati20141635} proving $n_0\ne 11$.
Even more recently, Gianfranco Casnati, Joachim Jelisiejew, and Roberto Notari
have shown that $n_0=14$ \cite[Thms A and B]{casnati_jelisiejew_notari_Hilbert_schemes_via_ray_families}.
That is, they prove that all Gorenstein schemes of length at most $13$ are smoothable.
This has been predicted by Anthony Iarrobino (private communication).
See also \cite{MR3225123} for a related partial result.

For the remainder of this section we work in characteristic $0$,
so the reader may take $n_0$ to be $14$.
Nevertheless, to emphasize that the particular value does not matter and in order to write statements
which are easier to generalize to other characteristics---and also because the proof that $n_0=14$ in characteristic $0$
was in fact found after a first version of this paper was posted---we continue to write $n_0$.
\begin{prop}\label{prop: cr n br gt n}
   Let $n = \dim V \ge n_0$, and $d \ge 2n-1$.
   Then there exist concise polynomials $F \in S^d V$ with $cr(F) =n$, but $br(F) > n$.
\end{prop}

\begin{proof}
   Let $R$ be any non-smoothable Gorenstein scheme of length $n$. 
   Embed $R \subset \PP V$ in a concisely independent way.
   Let $F \in \langle \veronese_d (R) \rangle$ be a general element.
   Then $F$ is not contained in $\langle \veronese_d (R') \rangle$ for any $R' \subsetneqq R$ by \cite[Lem.~3.5(iii)]{MR3092255}.
   By \cite[Cor.~2.7]{MR3092255} we cannot have $F \in \langle \veronese_d (Q) \rangle$ for any scheme $Q \subset \PP V$
      of length less than $n$, so $cr(F) = n$.
   For the same reason, since $R$ is not smoothable, there exists no smoothable scheme $Q \subset \PP V$ of length at most $n$,
      with $F \in \langle \veronese_d (Q) \rangle$.
   Were $br(F) \le n$, then there would have to be such a $Q$,
   see for example \cite[Prop.~11]{BGI}, \cite[Lem.~2.6]{MR3092255},
      or \cite[Prop.~2.5]{MR3121848}.
   Thus $br(F) > n$.
\end{proof}

\begin{remark}
We have seen that if $F$ has an equipotent apolar generator
then $F$ is a limit of direct sums.
And we have seen that if $F$ is a concise form in $n$ variables
and $F$ is a limit of direct sums of $n$ terms then $F$ has $n-1$
equipotent apolar generators.
Now, a form $F$ as in Proposition~\ref{prop: cr n br gt n}
has $n-1$ equipotent apolar generators by 
Theorem~\ref{thm: cactus rank n is a limit of direct sums},
so it is a limit of direct sums,
but it is not necessarily a limit of direct sums of $n$ terms.

Thus the closure of the locus of Fermat polynomials is contained
in the locus of forms with $n-1$ equipotent apolar generators,
but this containment can be strict.
\end{remark}

\begin{prop}
   Let $n < n_0$.
   Then every homogeneous concise polynomial $F \in S^d V$ with $cr(F) =n$
   has $br(F) = n$.
\end{prop}

\begin{proof}
   By \cite[Thm.~1.4(i)]{MR3121848} we have $br(F)\le n$. Since $F$ is concise we also have $br(F) \ge n$.
\end{proof}

We would like to introduce some terminology about zero-dimensional schemes.
\begin{defn}
   Suppose $\mathcal{R} \to B$ is a flat family of zero-dimensional schemes, $b\in B$ is a closed point,
     and $R=\mathcal{R}_b$ is the special fiber over $b$.
   We say $\mathcal{R} \to B$ is a \defining{cleaving of $R$} if
    the base $B$ is irreducible,
    the special fiber $R$ is supported at a single point, and
    the general fiber is not supported at a single point.
   If $R$ admits a cleaving, then we say $R$ is \defining{cleavable}.
   Otherwise, i.e., if $R$ is finite scheme supported at a single point that does not admit any cleaving, we say $R$ is \defining{uncleavable}.
\end{defn}

We remark that in \cite{casnati_jelisiejew_notari_Hilbert_schemes_via_ray_families} 
   cleavable schemes are called \emph{limit-reducible}, 
   and uncleavable schemes are called strongly non-smoothable.
In \cite{MR805334} a component of the Hilbert scheme containing uncleavable schemes is called an \emph{elementary component}. 
Note however that not every scheme which belongs to an elementary component is uncleavable, as this component intersects also other components of the Hilbert scheme.

\begin{lemma}\label{lemma_properties_of_cleaving}
   The following are elementary properties of cleavings and (un-)cleavable schemes. 
   \begin{enumerate}
     \item   A single reduced point is uncleavable. All other smoothable schemes are cleavable.
     \item   A general fiber of any cleaving of a zero dimensional Gorenstein scheme supported at a single point is a Gorenstein scheme. 
     (More generally, every deformation of a finite Gorenstein scheme is Gorenstein.)
     \item   Every non-smoothable Gorenstein scheme of length $n_0=14$ is local
     (that is, supported at a single point) and uncleavable.
     \item   Every Gorenstein scheme of length less than $n_0$ is cleavable, unless it is a single reduced point.
   \end{enumerate}
\end{lemma}
\begin{proof}
   The first property is clear.
   To be Gorenstein is an open condition on the Hilbert scheme, thus the second property follows. 
   
   Also the third property is straightforward --- if there exists a cleaving of a Gorenstein scheme of length $n_0$,
      then $R$ is a flat limit of a disjoint union of two shorter Gorenstein schemes.
   By definition of $n_0$, both shorter schemes must be smoothable.
   Thus $R$ is smoothable.
   
   The final property is clear, since all smoothable schemes are flat limits of disjoint points, hence they admit a cleaving.
\end{proof}

Note that not every non-smoothable Gorenstein scheme is uncleavable.
Potentially it could happen that every non-smoothable Gorenstein scheme of some fixed length
admits a cleaving to a disjoint union of two shorter schemes, at least one of which is non-smoothable.
Thus while every $n \geq n_0$ is clearly the length of a non-smoothable Gorenstein scheme,
it could potentially happen that some $n > n_0$ is not the length of any uncleavable Gorenstein scheme.
And in fact it is not immediately clear how to show that any $n > n_0$ is the length of an uncleavable Gorenstein scheme.
It would be rather surprising if $n_0$ were the only length of an uncleavable Gorenstein scheme,
or if there were only finitely many such lengths.
If that were the case, the Gorenstein schemes would be ``finitely generated'',
that is, there would be a finite number of schemes (``generators''),
such that all the others can be obtained in a flat family (with irreducible base), from a disjoint union of ``generators''.
It seems more plausible that every sufficiently large integer is the length of some uncleavable Gorenstein scheme,
or at least that there are infinitely many such lengths.
But it is beyond the scope of this paper to determine all the possible lengths of uncleavable Gorenstein schemes.
For the purpose of Theorem~\ref{thm: apolim and apoequ} it is enough that there exists such a length,
namely $n_0$.

Let us briefly remark, that there exist uncleavable schemes of any (sufficiently high) finite length, 
   if we drop the assumption of Gorenstein 
   \cite[Thm~2]{shafarevich_deformations_of_commutative_algebras_of_socle_degree_2}.
   In the Gorenstein case, it is expected that for every $n \ge 8$ 
     a general Gorenstein scheme with Hilbert function $(1,n,n,1)$ is uncleavable, see \cite[Lem.~6.21]{MR713096}.

So, let $R$ be an uncleavable Gorenstein scheme of length $n_1 > 1$.
In particular, by Lemma~\ref{lemma_properties_of_cleaving}, we have $n_1 \ge n_0$.
For example, we may choose $R$ such that $n_1 = n_0$.
In characteritic $0$, another possible value of $n_1 > 14$, 
   would be the minimal length of non-smoothable Gorenstein scheme contained in $\PP^5$ (or, respectively, in $\PP^4$).
It is known that such $n_1 \le 42$ (respectively, $n_1\le 140$).
\begin{prop}\label{prop: limit not apolar limit}
Suppose $n=n_1$ and $d \ge 2n_1$, with $R \subset \PP V$ a concisely independently embedded uncleavable Gorenstein scheme of length $n_1$.
   Let $F \in S^d V$ be a concise polynomial such that $cr(F) =n_1$ 
     and $F\in \langle \veronese_d(R) \rangle$.
   (For example, if $n_1=n_0$, take $F$ with $cr(F) =n_0$ and $br(F) > n_0$.)
   Then $F$ is a limit of direct sums, but it is not an apolar limit of direct sums.
\end{prop}
See Example~\ref{example_limit_polynomial_but_not_apolar_limit} for an explicit example of such a polynomial $F$.

\begin{proof}
   Suppose on the contrary that $F$ is an apolar limit of direct sums $F_t = G_t + H_t$.
   The Hilbert function of $A_F$ is $(1, n, n, \dotsc, n, 1,0,\dotsc)$ 
      by Lemma~\ref{lemma: cactus rank n and d ge 2n}.
   Consider the two Hilbert functions $h_{A_{G_t}}$ and $h_{A_{H_t}}$.
   We must have $h_{A_{G_t}}(k)+h_{A_{H_t}}(k) = n$ for all $1\leq k \leq d-1$.
   In particular, since $h_{A_{G_t}}(k),h_{A_{H_t}}(k) \ge 1$,
      we have $h_{A_{G_t}}(k),h_{A_{H_t}}(k) \leq n-1$, for all $1\leq k \leq d-1$.
   By \cite[Lem.~5.2]{MR3121848} we must have  
     $h_{A_{G_t}}(k) = (1,a,a,\dotsc, a,1,0,\dotsc)$
     and $h_{A_{H_t}}(k) = (1,b,b,\dotsc, b,1,0,\dotsc)$ for all $t$ close to $0$, where $a+b=n$.
  
   By \cite[Thm.~1.6]{MR3121848}, $G_t \in \langle \veronese_d(Q'_t)\rangle$ and $H_t \in \langle \veronese_d(Q''_t)\rangle$
     for schemes $Q'_t$ and $Q''_t$ of length $a$ and $b$, respectively.
   Denote the flat limit $Q =\lim_{t\to 0} (Q'_t \sqcup Q''_t)$.
   The length of $Q$ is $a+b=n$. 
   Moreover, for each $t$ close but not equal to zero, 
     $Q'_t$ and $Q''_t$ are embedded (respectively) into disjoint linear subspaces $\PP^{a-1}_t$ and $\PP^{b-1}_t$.
   Thus the defining ideal $I(Q'_t \sqcup Q''_t)$ of $Q'_t \sqcup Q''_t$ satisfies:
   \[
      I(Q'_t \sqcup Q''_t) = I(Q'_t) \cap I(Q''_t) = (G_t^{\perp})_{\le n} \cap (H_t^{\perp})_{\le n}
   \]
   The last equality $I(Q'_t)= (G_t^{\perp})_{\le n}$ (and analogously for $Q''_t$ and $H_t$)
      follows from \cite[Thm.~1.6(iii)]{MR3121848}.
   Since $F_t = G_t + H_t$ is a direct sum, 
      by Lemma~\ref{lem: apolar of direct sum} one has 
   $(G_t^{\perp})_{\le n} \cap (H_t^{\perp})_{\le n} = (F_t^{\perp})_{\le n}$. 
   Since we are considering an apolar limit, we may pass to the limit with ideals:
   \begin{align*}
     J & = \lim_{t\to 0} (I(Q'_t \sqcup Q''_t)) = \lim_{t\to 0} (F_t^{\perp})_{\le n}  = (F^{\perp})_{\le n}.
   \end{align*}
   $J$ is a homogeneous ideal defining $Q= \lim_{t\to 0} (Q'_t \sqcup Q''_t)$, 
     though at this point potentially $J$ is not saturated.
   However, using $J= (F^{\perp})_{\le n}$, and by \cite[Thm.~1.6(iii)]{MR3121848} applied to $F$,
     the ideal $J$ must be saturated and $F \in \langle \veronese_d(Q) \rangle$.
   By the uniqueness in \cite[Thm.~1.6(ii)]{MR3121848} we have $R = Q$, 
         which is a contradiction, since $Q$ is a limit of smaller disjoint schemes, 
         i.e., $Q$ is cleavable, while $R$ is not.
         
   In the case $n_1 =n_0$, $Q$ is smoothable of length $n_0$, 
      the above considerations imply that the border rank of $F$ is (at most) $n$,
      a contradiction with our assumption $br(F) > n$.
\end{proof}

\begin{example} \label{example_limit_polynomial_but_not_apolar_limit}
   Let $G$ be a general homogeneous cubic in $6$ variables $\fromto{x_1}{x_6}$,
     and let
   \[
     F = (d-2) z^{d-3} G + z^{d-2} (x_1 y_1 + x_2 y_2 + x_3 y_3 + x_4 y_4 + x_5 y_5 + x_6 y_6) + \frac{1}{d-1} z^{d-1} w.
   \]
   Consider the concisely independent scheme $R$ defined from $G$ as in Example~\ref{example_concise scheme}. 
   Then $F$ is apolar to $R$, which can be verified by acting on $F$ with the generators 
       \ref{item_generators_of_I__apolar_to_G}--\ref{item_generators_of_I__Theta_minus_beta_delta} 
       of Example~\ref{example_concise scheme}. 
   Thus  $F \in \langle \veronese_d(R) \rangle$.
   Moreover, since $G$ is general, in characteristic $0$, 
       $R$ is a shortest non-smoothable Gorenstein scheme, 
       see \cite[Thm A]{casnati_jelisiejew_notari_Hilbert_schemes_via_ray_families} and \cite[Lem.~6.21]{MR1735271}.
   In particular, $R$ is uncleavable by Lemma~\ref{lemma_properties_of_cleaving}. 
   Thus, by Proposition~\ref{prop: limit not apolar limit}, $F$ is a limit of direct sums,
       but not an apolar limit of direct sums.
\end{example}

\subsection{Apolar limits in the plane}

We show that in the plane, $\ApoLim = \ApoEqu$.
Recall that when considering forms in $3$ variables
we write $S = \C[x,y,z]$ and $T = \C[\alpha, \beta, \gamma]$.

\begin{thm}\label{thm: apolim=apoequ in the plane}
Let $F$ be a concise form of degree $d$ in $n=3$ variables
having an equipotent apolar generator.
Then $F$ is an apolar limit of direct sums.
\end{thm}

Note that in contrast to the situation of Proposition~\ref{prop: cubic apolar limit},
  where every family of direct sum cubic forms having as the limit a concise cubic form must be an
   apolar family, we
   do not claim here that every family $F_t \to F$ is necessarily an apolar family.
Rather, we claim only that there exists \textit{some} apolar family of direct sums $F_t \to F$.

As we will see in the proof of Theorem \ref{thm: apolim=apoequ in the plane}, 
   ``typically'' (here we do not want to specify precisely what does ``typically'' mean, see the proof below for an explicit statement),
   the limit indicated in Example~\ref{example: form of plane direct sums} provides an apolar limit.
However, this limit does not work in all cases, as illustrated by the following example:
\begin{example}
   Consider the following sextic in three variables:
   \[
     F = xy^5 + y^3 z^3.
   \]
   As indicated by Example~\ref{example: form of plane direct sums}, this is a limit of direct sums.
   Indeed:
   \[
    F_t = \tfrac{1}{t} \left( (y+\tfrac{1}{6} t x)^6 + t y^3 z^3 - y^6 \right)
   \]
   is a family of homogeneous polynomials ($t$ is a parameter of the family), with $F_0 = F$ 
     and $F_t$ for $t\ne 0$ after an easy coordinate change becomes:
   \[
     x^6 + y^3 z^3 - y^6.
   \]
   In particular, $F_t$ for $t\ne 0$ is a direct sum and the Hilbert function of $A_{F_t}$ is $1,3,4,5,4,3,1$,
      while the Hilbert function of $A_F = A_{F_0}$ is  $1,3,4,4,4,3,1$.
   Thus the family $F_t$ is not an apolar family.
   However, another family:
   \[
     \tfrac{1}{t} \left( (y+\tfrac{1}{6} t x)^6 - y^6 + t y^3 z^3 - \tfrac{1}{400} t^2 z^6 \right)
   \]
   presents $F$ as an apolar limit of direct sums.
\end{example}

\begin{proof}[Proof of Theorem~\ref{thm: apolim=apoequ in the plane}]
By Example~\ref{example: form of plane direct sums},
either $F$ is a direct sum, in which case the statement is trivial,
or else after a change of coordinates, $F = xy^{d-1} + G(y,z)$.
Write $G(y,z) = \sum_{q=0}^d \binom{d}{q} a_q y^q z^{d-q}$.

First we find the Hilbert function of $A_F$.
In order to reduce subscripts, we write $h_F$ for the Hilbert function of $A_F$.
We have $h_F(0) = h_F(d) = 1$.
We claim that for $1 \leq k \leq d-1$, $h_F(k) = h_{\gamma^2 \aa G}(k-1)+2$.

Recall that for a form $H \in S_e$ of degree $e$ and an integer $i \geq 0$,
$T_i \aa H$ is the linear subspace
$T_i \aa H = \{ \Theta \aa H \mid \Theta \in T_i \} \subseteq S_{e-i}$.

First, $h_{\gamma^2 \aa G}(k-1) = \dim T_{d-k-1} \gamma^2 \aa G$,
and $T_{d-k-1} \gamma^2 \aa G$ is spanned by, for $0 \leq i \leq d-k-1$,
\[
\begin{split}
  \beta^i \gamma^{d-k-1-i} \gamma^2 \aa G &= \frac{d!}{(k-1)!} \sum_{q=i}^{i+k-1} \binom{k-1}{q-i} a_q y^{q-i} z^{k+i-q-1} \\
    &= \frac{d!}{(k-1)!} \sum_{j=0}^{k-1} \binom{k-1}{j} a_{j+i} y^j z^{k-j-1} .
\end{split}
\]
So $h_{\gamma^2 \aa G}(k-1) = \rank M$ where $M_{ij} = \binom{k-1}{j} a_{j+i}$ for $0 \leq i \leq d-k-1$, $0 \leq j \leq k-1$.

Meanwhile $h_F(k) = \dim T_{d-k} \aa F$, and $T_{d-k} \aa F$ is spanned by $T_{d-k-1} \gamma \aa F = T_{d-k-1} \gamma \aa G$
together with $\{\alpha^j \beta^{d-k-j} \aa F \mid 0 \leq j \leq d-k\}$.
Note that $\alpha^2 \aa F = 0$.
We have
\[
  \alpha \beta^{d-k-1} \aa F = \frac{(d-1)!}{k!} y^k,
  \qquad
  \beta^{d-k} \aa F = \frac{(d-1)!}{(k-1)!} x y^{k-1} + \beta^{d-k} \aa G .
\]
Here, $\beta^{d-k} \aa F$ is linearly independent of the other spanning elements since it is the only one
with a monomial involving $x$.
And $T_{d-k-1} \gamma \aa G$ is spanned by, for $0 \leq i \leq d-k-1$,
\[
\begin{split}
  \beta^i \gamma^{d-k-1-i} \gamma \aa G &= \frac{d!}{k!} \sum_{q=i}^{i+k} \binom{k}{q-i} a_q y^{q-i} z^{k+i-q} \\
    &= \frac{d!}{k!} \sum_{j=0}^k \binom{k}{j} a_{j+i} y^j z^{k-j} .
\end{split}
\]
Thus $T_{d-k-1} \gamma \aa G + \C y^k$ is spanned by $y^k$ together with, for $0 \leq i \leq d-k-1$,
\[
  \sum_{j=0}^{k-1} \binom{k}{j} a_{j+i} y^j z^{k-j} .
\]
So $\dim(T_{d-k-1} \gamma \aa G + \C y^k) = 1 + \rank N$ where $N_{ij} = \binom{k}{j} a_{j+i}$
for $0 \leq i \leq d-k-1$, $0 \leq j \leq k-1$.
Note $N$ is obtained from $M$ by rescaling columns, so $\rank M = \rank N$.
This proves that $h_F(k) = h_{\gamma^2 \aa G}(k-1) + 2$, as claimed.

For any $1 \leq r \leq \frac{d}{2}$, let $h^r$ be the following function:
\[
  h^r(k) =
    \begin{cases}
      1, & k = 0 \\
      k+2, & 1 \leq k \leq r-1 \\
      r+2, & r \leq k \leq d-r \\
      (d-k)+2, & d-r+1 \leq k \leq d-1 \\
      1, & k = d \\
      0, & k<0 \text{ or } k>d .
    \end{cases}
\]
What we have shown is that $h_F = h^r$ where $r = br(\gamma^2 \aa G)$.
Let
\[
  H^r = \{ F = xy^{d-1} + G(y,z) \mid br(\gamma^2 \aa G) = r \} .
\]
Write $\hat\sigma_r(\veronese_d(\PP^1))$ for the affine cone over the $r$-th secant variety.
Then $F \mapsto \gamma^2 \aa G$ maps $H^r$ onto $\hat\sigma_r(\veronese_{d-2}(\PP^1)) \setminus \hat\sigma_{r-1}(\veronese_{d-2}(\PP^1))$,
the set of binary forms of border rank $r$.
Also the fibers of the map are irreducible, specifically copies of $\A^2$,
as the fiber through $F$ consists of $F + ay^d + by^{d-1} z$.
Thus $H^r$ is irreducible.

The claim of the Theorem is that for all $F \in H^r$ we can obtain $F$ as a limit of direct sums which have Hilbert function $h^r$. In the following paragraph we are going to prove the claim under an additional assumption that 
  $F$ is a general element of $H^r$. Then, in the final paragraph, we are going to use this ``generic'' case, and the irreducibility of $H^r$ to conclude the statement for all polynomials $F$ in question.

So suppose $F \in H^r$ is general.
Then $\gamma^2 \aa G \in \hat\sigma_r(\veronese_{d-2}(\PP^1))$ is also a general element,
so $\gamma^2 \aa G = \ell_1^{d-2} + \dotsb + \ell_r^{d-2}$ where the $[\ell_i] \in \PP^1$ are in general position.
In this case it is easy to integrate $\gamma^2 \aa G$ (twice) and
hence $G = c_1 \ell_1^d + \dotsb + c_r \ell_r^d + ay^{d-1}z + by^d$.
Therefore $F = (x+az+by)y^{d-1} + (c_1 \ell_1^d + \dotsb + c_r \ell_r^d)$.
Let $x' = x+az+by$.
For $t \neq 0$, let $F'_t = \frac{1}{t}(y + \frac{t}{d} x')^d + (c_1 \ell_1^d + \dotsb + c_r \ell_r^d - \frac{1}{t} y^d)$.
Since the $\ell_i$ are general and $r+1 \leq \frac{d+2}{2}$,
we have for general $t \neq 0$,
\[
  r \left( c_1 \ell_1^d + \dotsb + c_r \ell_r^d - \frac{1}{t}y^d \right)
  = br \left( c_1 \ell_1^d + \dotsb + c_r \ell_r^d - \frac{1}{t}y^d \right)
  = r+1 .
\]
Since $F'_t$ is a direct sum, $h_{F'_t} = h^r = h_F$
    (the first equality follows from Proposition~\ref{prop: direct sum poly -> connected sum apolar algebra}).
Thus $F'_t \to F$ is an apolar limit.

By the argument in the previous paragraph, the irreducible set $H^r$ is contained in 
   the Zariski closure of the locus of direct sums with Hilbert function $h^r$.
So every $F \in H^r$ is a limit of direct sums with Hilbert function $h^r$, i.e.,
  it is an apolar limit of direct sums. This proves the claim of the theorem.
\end{proof}

Note that the apolar ideal of a form in $3$ variables is a height $3$ Gorenstein ideal,
and is therefore generated by the principal Pfaffians of a skew-symmetric matrix \cite{MR0453723}.
Nevertheless we have not used this information;
instead, the key was information about apolar ideals of forms in one less variable.
It would be interesting to investigate whether the structure described in \cite{MR0453723}
can lead to a generalization of Theorem~\ref{thm: apolim=apoequ in the plane} for forms in $n=4$ variables,
  compare with \cite{elkhoury_srinivasan_a_class_of_Gorenstein_Artin_algebra_of_codim_4}.

It would also be interesting to study limits of direct sums of type $(1,n-1)$,
$s$-fold direct sums of type $(1,1,\dotsc,1,n-s+1)$,
direct sums of type $(1,\dotsc,1,2,\dotsc,2)$, and so on.
Note however that for limits of direct sums of type $(1,n-1)$ one cannot expect a similar result to
   Theorem~\ref{thm: apolim=apoequ in the plane}. This is because for $n=14$, 
   the polynomial $F$ presented in Example~\ref{example_limit_polynomial_but_not_apolar_limit}
   is a limit of direct sums of type $(1,13)$ and it has been proved it is not an apolar limit.

\section{Generalizations and further questions}\label{sect: generalizations}

\subsection{Linear series}\label{sect: linear series}

There are natural generalizations of some of our results to linear series.
Let $W \subseteq S^d V$ be a linear series.
A \defining{simultaneous power sum decomposition} of $W$ is a collection of linear forms $\ell_1,\dotsc,\ell_r$
such that $W$ is contained in the span of $\ell_1^d,\dotsc,\ell_r^d$;
equivalently, for each $F \in W$, there are scalars $c_1,\dotsc,c_r$ such that $F = \sum c_i \ell_i^d$.
The \defining{simultaneous Waring rank} of $W$, denoted $r(W)$,
is the least length $r$ of a simultaneous power sum decomposition.

Let $S = \C[V] = \C[x_1,\dotsc,x_n]$, where $x_1,\dotsc,x_n$ is a basis for $V$,
and let $T = \C[V^*] = \C[\alpha_1,\dotsc,\alpha_n]$ act on $S$ by letting $\alpha_i$ act as
the partial differentiation operator $\partial/\partial x_i$.
The \defining{apolar annihilating ideal} $W^\perp \subset T$ is $W^\perp = \{ \Theta \mid \Theta \aa F = 0, \forall F \in W \}$.
That is, $W^\perp = \bigcap_{F \in W} F^\perp$.
The apolar algebra $A_W = T / W^\perp$ is a level Artinian algebra with socle degree $d$ and type equal to the dimension of $W$,
meaning that its socle is entirely in degree $d$ and has dimension equal to $\dim W$.
In particular $A_W$ is Gorenstein if and only if $W$ is one-dimensional, i.e., spanned by a single form.

There is an Apolarity Lemma just as in the case of a single form.
This is well-known to experts; see for example \cite[Thm.~2.3]{MR2223453}.
For the reader's convenience we state it here:
\begin{lemma}
With notation as above, $\ell_1,\dotsc,\ell_r$ is a simultaneous power sum decomposition of $W$
if and only if the ideal $I = I(\{[\ell_1],\dotsc,[\ell_r]\})$ satisfies $I \subseteq W^\perp$.
\end{lemma}
\begin{proof}
$\ell_1,\dotsc,\ell_r$ is a simultaneous power sum decomposition of $W$
if and only if $W$ lies in the span of the $\ell_i^d$,
if and only if every $F \in W$ can be written as a linear combination of the $\ell_i^d$,
if and only if $I \subseteq F^\perp$ for every $F \in W$ (by the usual Apolarity Lemma),
if and only if $I \subseteq \bigcap F^\perp = W^\perp$.
\end{proof}

Thus $r(W) \geq \dim (A_W)_a$ for every $0 \leq a \leq d$: indeed, if $\ell_1,\dotsc,\ell_r$
is a simultaneous power sum decomposition of $W$ with defining ideal $I$ then
$r = \codim I_a  \geq \codim (W^\perp)_a = \dim (A_W)_a$ for $a > 0$ (and it is trivial for $a=0$).

There is a generalization of the Ranestad--Schreyer lower bound \cite{MR2842085} for Waring rank for linear series,
with essentially the same proof as for the case of a single form.
We briefly review the proof for completeness.
\begin{prop}
With notation as above, let $\length(A_W)$ be the length of $A_W$ and suppose $W^\perp$ is generated in degrees
less than or equal to $\delta$.
Then $r(W) \geq \length(A_W) / \delta$.
\end{prop}
\begin{proof}
Suppose $\ell_1,\dotsc,\ell_r$ is a simultaneous power sum decomposition of $W$ with defining ideal $I$.
The vanishing locus $V((W^\perp)_\delta)$ in affine space is just the origin (i.e., a scheme supported at the origin).
By Bertini's theorem, the linear series $(W^\perp)_\delta$ has no basepoints in projective space.
Let $G \in (W^\perp)_\delta$ be a general form.
Then $G$ does not vanish at any projective point $[\ell_i]$, so the affine hypersurface $V(G)$ does not contain
any line which is an irreducible component of $V(I)$.
Therefore by Bezout's theorem, the intersection of $V(G)$ and $V(I)$ has degree equal to $\delta r$.
But this intersection contains the scheme $V(W^\perp)$ which has length equal to $\length(A_W)$.
So $\delta r \geq \length(A_W)$.
\end{proof}

A \defining{direct sum decomposition} of $W$ is an expression $V = V_1 \oplus V_2$
and subspaces $W_1 \subseteq S^d V_1$, $W_2 \subseteq S^d V_2$,
such that $W \subset W_1 \oplus W_2$
and the projections $W \to W_1$, $W \to W_2$ are isomorphisms.
Equivalently, for a basis $F^1,\dotsc,F^k$ of $W$,
each $F^i = F^i_1 + F^i_2$ with $F^i_1 \in W_1$, $F^i_2 \in W_2$,
and the $F^i_1$ are linearly independent and so are the $F^i_2$.

Now we can generalize some of our results to the case of linear series.
Here is a generalization of Theorem~\ref{thm: direct sum generator}:
\begin{prop}
If a linear series $W$ of degree $d$ forms admits a direct sum decomposition then $W^\perp$
has at least $s = \dim W$ minimal generators of degree $d$.
\end{prop}
\begin{proof}
Say $W \subset W_x \oplus W_y$ is a direct sum decomposition where 
\[W_x \subseteq (S^x)_d = \C[x_1,\dotsc,x_i]_d, \ \
 W_y \subseteq (S^y)_d = \C[y_1,\dotsc,y_j]_d, \ \text{ and }W_x, W_y \neq 0.\]
We denote the dual rings $T^\alpha = \C[\alpha_1,\dotsc,\alpha_i]$, $T^\beta = \C[\beta_1,\dotsc,\beta_j]$,
$T = T^\alpha \otimes T^\beta$.

We have $W_x^\perp \cap W_y^\perp \subseteq W^\perp$.
If $0 \leq k \leq d$ and $\Theta \in (W^\perp)_k$ then for every $F \in W$,
say $F = G-H$ where $G \in W_x$, $H \in W_y$,
we have $\Theta \aa (G-H) = 0 = \Theta \aa G - \Theta \aa H$.
So $\Theta \aa G \in S^x_{d-k}$ and $\Theta \aa H \in S^y_{d-k}$ are equal,
which implies $\Theta \aa G = \Theta \aa H = 0$ or $d-k = 0$.
Thus $(W_x^\perp \cap W_y^\perp)_k = W^\perp_k$ for $0 \leq k < d$.

Let $F^1,\dotsc,F^s$ be a basis for $W$ and for each $i$ let $F^i = F^i_x - F^i_y$, $F^i_x \in W_x$, $F^i_y \in W_y$.
For $1 \leq j \leq s$ let $\delta_{x,j} \in T^x_d$ such that $\delta_{x,j} \aa F^i_x = 1$ if $i=j$, $0$ if $i \neq j$.
Similarly let $\delta_{y,j} \in T^y_d$ such that $\delta_{y,j} \aa F^i_y = 1$ if $i=j$, $0$ if $i \neq j$.
There are such elements by the linear independence of the $F^i_x$ and the $F^i_y$.
Let $\Delta_j = \delta_{x,j} + \delta_{y,j}$.
Then $\Delta_j \aa F^i = 0$ for each $i$, so $\Delta_j \in W^\perp$, but $\Delta_j \aa F^j_x = \Delta_j \aa F^j_y = 1$
so $\Delta_j \notin W_x^\perp \cap W_y^\perp$.
Hence each $\Delta_j$ is a minimal generator of $W^\perp$.
The $\Delta_j$ are linearly independent since if $\sum a_j \Delta_j = 0$ then $a_i = (\sum a_j \Delta_j) \aa F^i_x = 0$ for each $i$.
\end{proof}

Next we give a generalization of Proposition~\ref{prop: d+1 generator rank 1}.
\begin{prop}
Let $W \subseteq S^d V$.
Then $W^\perp$ has a minimal generator of degree $d+1$ if and only if there is a nonzero $(d+1)$-form $G \in S^{d+1} V$
such that $T_1 \aa G \subseteq W$.
\end{prop}
\begin{proof}
First suppose $W^\perp$ has a minimal generator of degree $d+1$.
Let $I \subset W^\perp$ be the ideal generated by elements of degree $\leq d$.
Then $I_{d+1} \neq T_{d+1}$.
Let $G \in S^{d+1} V$ be a nonzero element annihilated by $I_{d+1}$.
Since $G$ is annihilated by $I_{d+1}$, $G$ is annihilated by $I$, so $(W^\perp)_d = I_d \subseteq (G^\perp)_d$.
If $F \in T_1 \aa G$ then $G^\perp \subseteq F^\perp$.
In particular $(W^\perp)_d \subseteq F^\perp$, so $F \in W$.
This shows $T_1 \aa G \subseteq W$.

Conversely, suppose $G \in S^{d+1} V$ is a nonzero $(d+1)$-form such that $T_1 \aa G \subseteq W$.
As before, let $I \subseteq W^\perp$ be the ideal generated by elements of degree $\leq d$.
In particular, $I_d \subseteq (W^\perp)_d \subseteq ((T_1 \aa G)^\perp)_d$.
Then $I_{d+1} \subseteq (G^\perp)_{d+1}$: indeed, if $\Theta = \sum \alpha_i \theta_i$, each $\theta_i \in I_d$,
then $\Theta \aa G = \sum \theta_i \aa (\alpha_i \aa G) = 0$
since each $\alpha_i \aa G \in W$ and each $\theta_i \in I_d \subseteq (W^\perp)_d$.
In particular, $I_{d+1} \neq T_{d+1} = (W^\perp)_{d+1}$.
So $W^\perp$ has a minimal generator of degree $d+1$.
\end{proof}

Recall that the $r$-th \defining{prolongation} of $W \subseteq S^d V$
is the set of $(d+r)$-forms $G$ such that $T^{e-d} \aa G \subseteq W$ \cite[Definition~1.1]{MR2541390}.
The above Proposition shows that $W^\perp$ has a minimal generator of degree $d+1$ if and only if
the first prolongation of $W$ is nonzero.

We give a generalization of Theorem~\ref{thm: apolar generator degree lower bound}.
First we generalize Lemma~\ref{lemma: apolar containment}.
\begin{lemma}
Suppose $W \subseteq S^d V$, $U \subseteq S^e V$ are linear series such that $U^\perp \subseteq W^\perp$.
Then $e \geq d$ and $W \subseteq T_{e-d} \aa U$.
That is, for every $F \in W$, there are $G \in U$ and $\Theta \in T_{e-d}$ such that $F = \Theta \aa G$.
\end{lemma}
\begin{proof}
This follows by the inclusion-reversing part of \cite[Thm.~21.6]{eisenbud:comm-alg}.
\end{proof}
Now here is a generalization of Theorem~\ref{thm: apolar generator degree lower bound}.
\begin{prop}
Let $W \subseteq S^d V$ be a linear series and $n = \dim V$.
Suppose $W^\perp$ has minimal generators $\Theta_1,\dotsc,\Theta_s$ such that
$\deg \Theta_i = d_i$ for each $i$, and $d_1 \leq \dotsb \leq d_s$.
Let $\delta$ be an integer such that the ideal $(W^\perp)_{\leq \delta}$ is
$\gotm$-primary.
Assume $d_k = \delta < d_{k+1}$ or $k = s$ and $\delta = d_s$; necessarily $k \geq n$.
Then $d \leq d_k + d_{k-1} + \dotsb + d_{k-n+1} - n$.
\end{prop}
The proof is similar to the proof of Proposition~\ref{prop: apolar generator degree lower bound strengthening}.
As before, we immediately deduce
\begin{cor}
Let $W \subseteq S^d V$ be a linear series.
Let $n = \dim V$.
Suppose $W^\perp$ is generated in degrees less than or equal to $\delta$.
Then $d \leq (\delta-1)n$.
\end{cor}

It would be interesting to see if our other results, such as Theorem~\ref{thm: apoequ = closure of dirsum},
can be generalized to linear series.
The proofs we have given have used Gorenstein duality,
which is not available since the ideals $W^\perp$ are not Gorenstein.

\subsection{Overlapping sums}
Let $F = G_1 - G_2$, $G_i \in S^d V_i$, $G_i \neq 0$ for $i=1,2$.
Theorem~\ref{thm: direct sum generator} shows that if $V_1 \cap V_2 = \{0\}$,
then $F^\perp$ has a minimal generator of degree $d$.
Here we are interested in allowing $V_1 \cap V_2$ to be nonzero,
so that they form an ``almost direct sum'' or ``overlapping sum.''
We give a statement for the case $\dim(V_1 \cap V_2) = 1$.

\begin{prop}
Let $F = G_1 - G_2$, $G_i \in S^d V_i$, with $G_i$ concise in $V_i$ for $i=1,2$,
and suppose $V_1 \cap V_2$ is one-dimensional, spanned by $x$.
Moreover, suppose $V_i \ne \langle x \rangle$, and $\max(\deg_x(G_1),\deg_x(G_2)) < d/2$.
Let $s = \max\{t \mid x^t \in T \aa G_1 \cap T \aa G_2\}$.
Then $F^\perp$ has a minimal generator of degree $d-s$.
\end{prop}

\begin{proof}
As in the proof of Theorem~\ref{thm: direct sum generator}, $G_1^\perp \cap G_2^\perp \subset F^\perp$.
We claim $F^\perp_a = (G_1^\perp \cap G_2^\perp)_a$ for $0 \leq a < d-s$.
If $\Theta \in F^\perp_a$ then $\Theta \aa G_1 = \Theta \aa G_2 \in S^{d-a} \langle x \rangle$,
that is, $\Theta \aa G_1 = \Theta \aa G_2 = c x^{d-a}$ for some scalar $c$.
We have $x^s \in T \aa G_1 \cap T \aa G_2$, $x^{s+1} \notin T \aa G_1 \cap T \aa G_2$;
more generally, $x^k \in T \aa G_1 \cap T \aa G_2$ if and only if $k \leq s$.
So we must have $c=0$ or $d-a \leq s$, equivalently $\Theta \in G_1^\perp \cap G_2^\perp$ or $a \geq d-s$.
This proves the claim.

Note if $x^t \in T \aa G_1$ then $t \leq \deg_x(G_1)$.
Thus $s \leq \deg_x(G_1) < d/2$.

Now there exist $\delta_1, \delta_2 \in T_{d-s}$ such that $\delta_i \aa G_i = x^s$, $\delta_1 \aa G_2 = \delta_2 \aa G_1 = 0$.
To see this, let $V_1$ have basis $\{x,y_1,\dotsc,y_j\}$ and let $V_2$ have basis $\{x,z_1,\dotsc,z_k\}$.
Let $\{\alpha,\beta_1,\dotsc,\beta_j\}$ be the dual basis for $V_1^*$ and
$\{\alpha,\gamma_1,\dotsc,\gamma_k\}$ be the dual basis for $V_2^*$.
Let $T^\beta = \C[\alpha,\beta_1,\dotsc,\beta_j]$ and $T^\gamma = \C[\alpha,\gamma_1,\dotsc,\gamma_k]$.
We have $T = \C[\alpha,\beta_1,\dotsc,\beta_j,\gamma_1,\dotsc,\gamma_k]$.
There is a $\delta'_1 \in T_{d-s}$ such that $\delta'_1 \aa G_1 = x^s$.
Since every term of $\delta'_1$ that involves a $\gamma_i$ annihilates $G_1$,
we can delete those terms to get an element $\delta_1 \in T^\beta_{d-s}$ such that $\delta_1 \aa G_1 = x^s$.
Every term of $\delta_1$ that has a $\beta_i$ annihilates $G_2$.
If $\delta_1$ has a term $c \alpha^{d-s}$, $c \neq 0$, then by the hypothesis $\deg_x(G_2) < d/2 < d-s$
we see that this term also annihilates $G_2$.
So $\delta_1 G_2 = 0$ as desired.
It is similar to produce $\delta_2$.

Let $\Delta = \delta_1 + \delta_2$.
Then $\Delta \in F^\perp_{d-s}$, but $\Delta \notin G_1^\perp \cap G_2^\perp$.
Hence $\Delta$ is a minimal generator of $F^\perp$: it cannot be generated in lower degrees,
since all elements in lower degrees lie in $G_1^\perp \cap G_2^\perp$.
\end{proof}

\begin{cor}
Let $F = G_1 - G_2$, $G_i \in S^d V_i$, $G_i$ concise in $V_i$ for $i=1,2$,
and suppose $V_1 \cap V_2$ is one-dimensional, spanned by $x$.
Moreover, suppose $V_i \ne \langle x \rangle$
and $t = \max(\deg_x(G_1),\deg_x(G_2)) < d/2$.
Then $F^\perp$ has a minimal generator of degree at least $d-t$, in particular strictly greater than $d/2$.
\end{cor}

It would also be interesting to investigate cases with larger overlap,
or more than two overlapping summands.

\subsection{Other base fields}\label{sect: other fields}

Most of our results, including results overlapping with \cite{KleppePhD2005},
also hold over any algebraically closed field $\kk$.
The main difference is that we need to consider $S= \kk[\fromto{x_1}{x_n}]^{DP}$ to be the divided power algebra, 
  rather than the polynomial ring, and the apolarity action of $T$ on $S$ is now as if the differentiation was very naive: 
  $\alpha_i \aa {x_i}^{(d)} = {x_i}^{(d-1)}$ (no $d$ coeffficient).
See \cite[Appendix A]{MR1735271}, \cite[Section A2.4]{eisenbud:comm-alg}.
All occurences of powers of linear forms in $S$ should now be replaced by the divided powers, 
  for instance $x^{d-1}y$ should be $x^{(d-1)}y$, etc.
In particular, the Veronese embedding $\veronese_d\colon \PP V \to \PP (S^d V) $ is now $\veronese_d([x])= [x^{(d)}]$
  and the Waring rank  is computed with respect to the image of $\veronese_d$ (and the border rank, cactus rank, etc.).
  
In this setup 
  Theorems~\ref{thm: direct sum generator}, \ref{thm: apoequ = closure of dirsum}
  and \ref{thm: apolar generator degree lower bound} and their proofs
  remain valid over any algebraically closed base field $\kk$ of any characteristic,
  so in particular, direct sums and their concise limits have equipotent apolar generators;
  any concise divided powers form with an equipotent apolar generator is a limit of direct sums;
  and the bound $d \leq (\delta-1)n$ relating the greatest degree $\delta$ of apolar generators and the degree $d$ of the form is valid.
Proposition~\ref{prop: d+1 generator rank 1} only requires the change of $\ell^d$ into $\ell^{(d)}$,
  so that forms of degree $d$ with apolar generator in degree $d+1$ essentially depend on just one variable.
Proposition~\ref{prop: smith stong} is proved over any field in \cite{MR2738376}.

Theorem~\ref{thm: apolim and apoequ} consists of two parts.
The first consists of Theorem~\ref{thm: apolim=apoequ in the plane} 
  (in three variables all limits of direct sums are apolar limits of direct sums)
  and Proposition~\ref{prop: cubic apolar limit} (all limits of cubics are apolar limits),
  which are valid with no significant change to the statements or proofs.
In fact the proof of Theorem~\ref{thm: apolim=apoequ in the plane} becames slightly simpler,
  as all coefficients like $\frac{d!}{k!}$ or $\binom{k}{q-i}$ are replaced with just $1$ 
  and in the end the matrices $M$ and $N$ are just equal.
The second part is Proposition~\ref{prop: limit not apolar limit}
  (there exists a limit of direct sums, which is not an apolar limit of direct sums). 
To prove this statement, we used results from \cite{MR3121848}, which is written over $\C$,
  so Proposition~\ref{prop: limit not apolar limit} is not proven over $\kk$.
Similarly, the other results presented in Sections~\ref{sect: cactus rank} and \ref{sect: non-apolar limit}
  also depend on \cite{MR3121848}.
However, the first and the second named authors believe that the results of \cite{MR3121848} used in this article
  can be generalized to any characteristic.
  
Section~\ref{sect: plane cubics} and Table~\ref{table:plane cubics} describe the behavior and classification of plane cubics. 
In positive characteristics, these are different, particularly the cases of characteristics $2$ and $3$.
The numerous examples throughout the paper might be valid only in some characteristics, 
      while in the other characteristics they need to be appropriately adjusted.
The exact values of integers $n_0$ (the length of the shortest non-smoothable scheme) 
      and $n_1$ (the possible lengths of uncleavable schemes) may be different in positive characteristics, 
      particularly in characteristics $2$ and $3$.
No other changes are needed to make this paper valid over any algebraically closed field.

\bigskip

\renewcommand{\MR}[1]{{}}


\providecommand{\bysame}{\leavevmode\hbox to3em{\hrulefill}\thinspace}


\bigskip

\end{document}